\documentclass{amsart}
\usepackage{appendix}
\usepackage{tabularx}
\usepackage[numbers,sort&compress]{natbib}
\usepackage{color}

\newtheorem{theorem}{Theorem}[section]
\newtheorem{lemma}[theorem]{Lemma}
\newtheorem{corollary}[theorem]{Corollary}

\theoremstyle{definition}

\theoremstyle{remark}
\newtheorem{remark}[theorem]{Remark}

\numberwithin{equation}{section}
\allowdisplaybreaks[4]

\begin{document}

\title{Curvature pinching estimate under the Laplacian $G_{2}$ flow}


\author{Chuanhuan Li}
\address{School of Mathematical Sciences, Laboratory of Mathematics and Complex Systems, Beijing Normal University, Beijing 100875, China \newline
${\quad}$Shanghai Institute for Mathematics and Interdisciplinary Sciences, Shanghai 200433, China}
\curraddr{}
\email{chli@mail.bnu.edu.cn,chli@simis.cn}
\thanks{}

\author{Yi Li$^{\ast}$}
\address{Shanghai Institute for Mathematics and Interdisciplinary Sciences, Shanghai 200433, China \newline
${\quad}$ Fudan University, Shanghai 200433, China}
\curraddr{}
\email{yilicms@simis.cn, yilicms@gmail.com}
\thanks{$^{\ast}$Corresponding author}



\date{}

\dedicatory{}

\begin{abstract}
In this paper, we derive a pinching estimate {for} the traceless Ricci curvature in terms of scalar curvature and the $C^{1}$  norm of the Weyl tensor under the Laplacian $G_{2}$ flow for closed $G_{2}$ structures. Then we apply this estimate to study the long time existence of the Laplacian $G_{2}$ flow and prove that the $C^{1}$  norm of the Weyl tensor has to blow up at least at a certain rate under bounded scalar curvature.

\end{abstract}

\maketitle

\section{Introduction}

In order to study a smooth $7$-manifold $M$ admitting a {torsion-free} $G_{2}$-structure, Bryant \cite{Bryant 2006} introduced the following Laplacian $G_2$ flow for closed $ G_{2}$-structure:
\begin{equation}
  \left \{
       \begin{array}{rl}
          \partial_{ t}\varphi(t)&=\Delta_{\varphi(t)}\varphi(t),\\
           \varphi(0)&=\varphi,
       \end{array}
  \right.
  \label{The closed Laplacian flow}
\end{equation}
where $\Delta_{\varphi(t)}\varphi(t)=dd^{\ast}_{\varphi(t)}\varphi(t)+d^{\ast}_{\varphi(t)}d\varphi(t)$ is the Hodge Laplacian of $g(t)$ and $\varphi$ is an initial  closed $G_{2}$-structure({a positive closed three-form).} Here $g(t)$ is the Riemannian metric {algebraically determined by} $\varphi(t)$. Since $\Delta_{\varphi}\varphi=dd^{\ast}_{\varphi}\varphi$ for a closed $G_{2}$-structure $\varphi$, we see that the closedness of $\varphi(t)$ is preserved along the Laplacian flow $\eqref{The closed Laplacian flow}$. 

The flow $\eqref{The closed Laplacian flow}$ can be viewed as the gradient flow for the Hitchin functional {when the variations are restricted to the cohomology class of the closed $G_{2}$-structure, where Hitchin functional} introduced by Hitchin \cite{Hitchin 2000}
$$\mathcal{H}:[\overline{\varphi}]_{+}\longrightarrow\mathbb{R}^{+},\ \varphi\longmapsto\frac{1}{7}\int_{M}\varphi\wedge\psi=\int_{M}\ast_{\varphi}1.$$
Here $\overline{\varphi}$ is a closed $G_{2}$-structure on $M$ and  $[\overline{\varphi}]_{+}$ is the open subset of the cohomology class $[\overline{\varphi}]$ consisting of $G_{2}$-structures. Any critical point of $\mathcal{H}$ gives rise to a torsion-free $G_{2}$-structure {when $7$-manifold $M$ is compact.} 

For more details about Laplacian flow on some special $7$-manifolds, Laplacian solitons and other flows on $G_{2}$-structures, see \cite{Bryant-Xu 2011, Chen Shi-estimates, Fer-fin 2016, Fine-Yao 2018, Fino 2021, Grigorian-2013, Grigorian-2016, Huang-wang-yao 2018, KL 2021, LLX25,local curvature of G2, Lin-2013, Lotay-Wei Shi-estimate, L-W 2019, L-W 2019 2, Lotay 2022, Nico 2022, Weiss-Witt 2012, Weiss-Witt 2012b}.

The stationary points of the Laplacian flow $\eqref{The closed Laplacian flow}$ are harmonic $\varphi$, which on a compact manifold are the torsion-free $G_{2}$-structure. Hence the long time behaviour of the flow is an important topic. In \cite{Bryant-Xu 2011}, applying De Turck's trick and Hamilton's Nash-Moser inverse function theorem, Bryant and Xu proved the short time existence for $\eqref{The closed Laplacian flow}$. In 2017, Lotay and Wei \cite{Lotay-Wei Shi-estimate} proved the long time existence for the Laplacian $G_{2}$ flow $\eqref{The closed Laplacian flow}$ under bounded Riemann curvature or Ricci curvature. In 2021, the second author \cite{local curvature of G2} gave {an alternative proof via local curvature estimates, showing that the flow persists as long as the Ricci curvature remains bounded.}

The analogue of long time behaviour of Ricci flow was proved by Hamilton \cite{Hamilton 1982} for bounded Riemann curvature and \u{S}e\u{s}um \cite{Sesum 2005} for bounded Ricci curvature (Kotschwar-Munteanu-Wang \cite{K-M-Wang 2016} gave another proof in this case).  The Riemann curvature gap estimate and the Ricci curvature gap estimate for finite time singular solutions of Ricci flow were given, respectively, by Lemma 8.7 in \cite{C-L-N Ricci flow} and Theorem 1.1 in \cite{Wang 2012}(another proof for Ricci curvature gap estimate was given by \cite{K-M-Wang 2016}). {Whether} the Ricci flow will exist as long as the scalar curvature remains bounded is still an open question. For the K\"ahler-Ricci flow \cite{zhang-2010} or type-I Ricci flow \cite{Ender-Top 2011}, this conjecture was settled. For the general case, some partial results were carried out in \cite{Curvature pinching by Cao}, {and Li \cite{Li24} weakened the assumption of \cite{Curvature pinching by Cao}.} Similarly, the analogue of long time behaviour for Ricci-harmonic flow, see \cite{Cheng-zhu 2013, RHF with bounded R I, RHF with bounded R, local curvature of RHF, List-2008, Muller-2012}, while for other flows, see \cite{LY, L2022}.

${}$

Lotay and Wei \cite{Lotay-Wei Shi-estimate} propose the following question (see also \cite{Lotay, Wei}):
\begin{quote}
    {\bf Question 1:} {\it whether the Laplacian $G_{2}$ flow on closed $G_{2}$-structures will exist as long as the torsion tensor or scalar curvature remains uniformly bounded.}
\end{quote}
Because the flow $\eqref{The closed Laplacian flow}$ preserves the closedness of $\varphi(t)$, we see that the scalar curvature of $g(t)$ is always nonpositive. Hence, the question of Lotay and Wei is equivalent to 
\begin{quote}
    {\bf Question 2:} {\it whether the Laplacian $G_{2}$ flow on closed $G_{2}$-structures will exist as long as the scalar curvature remains uniformly bounded from below.} 
\end{quote}
Fine and Yao \cite{Fine-Yao 2018} obtained the long time existence of the Laplacian $G_{2}$ flow $\eqref{The closed Laplacian flow}$ with bounded scalar curvature when $M=X\times \mathbb{T}^{3}$, where $X$ is a $4$-manifold with hypersymplectic structure and $\mathbb{T}^{3}$ is the standard flat 3-torus. Picard and Suan \cite{Picard-Suan} proved that the Laplacian $G_{2}$ flow $\eqref{The closed Laplacian flow}$ exists for all time when $M=\mathbb{T}^{3}\times X^{4}$ or $M=\mathbb{S}^{1}\times X^{6}$ with $X^{4},X^{6}$ are compact Calabi–Yau $2,3$-folds. For more general $7$-manifolds, the second author \cite{local curvature of G2} computed the evolution equation for the scalar curvature, giving a weak bound for the scalar curvature which can not be used to answer Lotay-Wei's question.

In this paper, we will give some partial results about {\bf Question 1} or {\bf Question 2}. Using the curvature pinching estimates method in \cite{Curvature pinching by Cao, Hamilton 1982} and  maximum principle, we have

\begin{theorem}\label{theorem1.1}
    Let $(M,\varphi(t))_{t\in[0,T)}$ be the solution of the Laplacian $G_{2}$ flow $\eqref{The closed Laplacian flow}$ on a closed $7$-dimensional manifold $M$ with $T<+\infty$, where $g(t)$ is the Riemannian metric associated with $\varphi(t)$. If there exists a positive constant $c$ independent on time $t$ such that $R(g(t))+c>0$, then there exist some constants {$C_{1}=C_{1}(c,\varphi(0))\geq0$} and $C_{2}=C_{2}(c)\geq0$, such that for all $t\in[0,T)$, one has 
    \begin{align}
        \frac{\displaystyle{\left|{\rm Ric}(g(t))-\frac{1}{7}R(g(t))\cdot g(t)\right|_{g(t)}}}{R(g(t))+c}\leq C_{1}+C_{2}\max_{M\times[0,t]}\frac{|W(g(t))|_{C^{1}(M,g(t))}}{R(g(t))+c}.
    \end{align}
    where  $W(g(t))$ is the Weyl tensor of $g(t)$.
\end{theorem}

{In fact, Cleyton and Ivanov \cite{CI08} showed that the  Weyl tensor can be decomposed into the following three parts by $G_{2}$ structure:(see also \cite{DGK25} for details)
$$W=W_{77}+W_{64}+W_{27}.$$
Thus, it is worth investigating which part, together with the scalar curvature,  controls the Ricci curvature in Theorem \ref{theorem1.1}.}

{From Theorem \ref{theorem1.1}, we give a partial result about {\bf Question 1} or {\bf Question 2}:}
\begin{theorem}{\bf (Partially answer {\bf Question 1} or {\bf 2})}\label{theorem1.2}
     Let $(M,\varphi(t))_{t\in[0,T)}$ be the solution of the Laplacian $G_{2}$ flow $\eqref{The closed Laplacian flow}$ on a closed $7$-dimensional manifold $M$ with $T<+\infty$, where $g(t)$ is the Riemannian metric associated with $\varphi(t)$. Either one has
     $$\liminf_{t\rightarrow T}\left(\min_{M} R(g(t))\right)=-\infty$$
     or (for some positive constant $c$)
     $$R(g(t))+c>0 \ \text{on} \ M\times[0,T) \ \ \ \text{but}\ \ \limsup_{t\rightarrow T}\left(\max_{M}\frac{|W(g(t))|_{C^{1}(M,g(t))}}{R(g(t))+c}\right)=+\infty.$$
\end{theorem}

\begin{remark}\label{remark1.3}
    From \cite{local curvature of G2,Lotay-Wei Shi-estimate} or (2.15), we get $R(g(t))=-|\mathbf{T}(t)|^{2}_{g(t)}\leq0$, then we can use the torsion form $\mathbf{T}(t)$ to replace $R(g(t))$.
\end{remark}

\begin{corollary}\label{corollary1.4}
    Let $(M,\varphi(t))_{t\in[0,T)}$ be the solution of the Laplacian $G_{2}$ flow $\eqref{The closed Laplacian flow}$ on a closed $7$-dimensional manifold $M$ with $T<+\infty$, where $g(t)$ is the Riemannian metric associated with $\varphi(t)$. If the scalar curvature $R(g(t))$ and the $C^{1}$  norm of Weyl tensor  $W(g(t))$ are both uniformly bounded, then the solution $\varphi(t)$ can be extended past time $T$.
\end{corollary}

Besides this, we study the blow-up rate of the $C^{1}$ norm for the Weyl tensor.

\begin{theorem}{\bf (Blow-up rate of the Weyl tensor)}\label{theorem1.5}
     Let $(M,\varphi(t))_{t\in[0,T)}$ be the solution of the Laplacian $G_{2}$ flow $\eqref{The closed Laplacian flow}$ on a closed $7$-dimensional manifold $M$ with $T<+\infty$, where $g(t)$ is the Riemannian metric associated with $\varphi(t)$. Then we have 

    $(1)$ either $\displaystyle{\liminf_{t\rightarrow T}\left((\min_{M} R(g(t))\right)=-\infty}$,

    $(2)$ or $\displaystyle{\liminf_{t\rightarrow T}\min_{M} R(g(t))>-\infty}$, but for any positive constants $C>0$ and $\delta>0$,
    $$\limsup_{t\rightarrow T}\left(\max_{M}|W(g(t))|_{C^{1}(M,g(t))}\right)>\frac{C}{(T-t)^{1-\delta}}.$$
\end{theorem}

We give an outline of this paper. We review the basic theory in Section \ref{section2} about $G_{2}$-structure, $G_{2}$-decompositions of $2$-forms and $3$-forms, the torsion tensors of $G_{2}$-structures, and curvature tensor of closed $G_{2}$-structures. Section 3 gives some evolution equations under the Laplacian $G_{2}$ flow $\eqref{The closed Laplacian flow}$. In Section 4, using curvature pinching estimates and the maximum principle, we prove the main results.

\section{$G_{2}$-structure}\label{section2}
In this section, we give some basis theory of $G_{2}$-structures in \cite{Bryant 2006,Spiros 2003,Spiros 2005,Spiros flow of G2,Spiros notes of G2 on 2010,Spiros introduce to G2,local curvature of G2,Lotay-Wei Shi-estimate}.

\subsection{$G_{2}$-structure on smooth $7$-manifold}\label{subsection2.1}

Let $\{e_{1},e_{2},\cdots,e_{7}\}$ denote the standard basis of $\mathbb{R}^{7}$ and let $\{e^{1},e^{2},\cdots,e^{7}\}$ be its dual basis. Define the 3-form
$$\phi:=e^{123}+e^{145}+e^{167}+e^{246}-e^{257}-e^{347}-e^{356},$$
where $e^{ijk}:=e^{i}\wedge e^{j}\wedge e^{k}$. The subgroup of ${\rm GL}(7,\mathbb{R})$ {that fixes $\phi$} is the exceptional Lie group $G_{2}$, a compact, connected, simple $14$-dimensional Lie subgroup of ${\rm SO}(7)$. In fact, $G_{2}$ acts irreducibly on $\mathbb{R}^{7}$ and preserves the metric and orientation for which $\{e_{1},e_{2},\cdots,e_{7}\}$ is an oriented orthonormal basis. Note that $G_{2}$ also preserves the $4$-form
$$\ast_{\phi}\phi=e^{4567}+e^{2367}+e^{2345}+e^{1357}-e^{1346}-e^{1256}-e^{1247},$$
where $\ast_{\phi}$ is the Hodge star operator determined by the metric and orientation.

For a smooth $7$-manifold $M$ and a point $x\in M$, we define as in \cite{local curvature of G2,Lotay-Wei Shi-estimate}
\begin{align}
    \wedge_{+}^{3}(T_{x}^{\ast}M):=\left\{\varphi_{x}\in\wedge^{3}(T_{x}^{\ast}M)\Big{|} u^{\ast}\phi=\varphi_{x},
    \text{for\ invertible}\ u\in\text{Hom}_{\mathbb{R}}(T_{x}^{\ast}M,\mathbb{R}^{7})\right\}\notag
\end{align}
and the bundle
\begin{align}
    \wedge_{+}^{3}(T^{\ast}M):=\bigcup_{x\in M}\wedge_{+}^{3}(T_{x}^{\ast}M)\notag.
\end{align}
We call a section $\varphi$ of $\wedge_{+}^{3}(T^{\ast}M)$ a {\it positive $3$-form} on $M$ or a {\it $G_{2}$-structure} on $M$, and denote the space of positive 3-form by $\Omega^{3}_{+}(M)$. The existence of $G_{2}$-structure is equivalent to the
property that $M$ is {both orientable and spinnable}, which is equivalent to the vanishing of the first and two Stiefel-Whitney classes $\omega_{1}(TM)$ and $\omega_{2}(TM)$. For more details, see Theorem 10.6 in \cite{Lawson-Michelsohn}.

For a {positive} $3$-form $\varphi$, we define a $\Omega^{7}(M)$-valued bilinear form $\text{B}_{\varphi}$ by
$$\text{B}_{\varphi}(u,v)=\frac{1}{6}(u\lrcorner\varphi)\wedge(v\lrcorner\varphi)\wedge\varphi,$$
where $u, v$ are tangent vectors on $M$ and $``\lrcorner"$ is the interior multiplication operator (\textit{Here we use the orientation in \cite{Bryant 2006}}). Then we can see that any $\varphi\in\Omega^{3}_{+}(M)$ determines a Riemannian metric $g_{\varphi}$ and an orientation $d V_{\varphi}$, hence the Hodge star operator $\ast_{\varphi}$ and the associated $4$-form
$$\psi:=\ast_{\varphi}\varphi$$
can also be uniquely determined by $\varphi$. 
{
\begin{remark}
    Throughout this paper, we consistently adopt the orientation induced by $G_{2}$-structure from \cite{Bryant 2006}.
\end{remark}}

If we choose local coordinates $\{x^{1},\cdots,x^{7}\}$ on $M$, then we can write $\varphi,\ \psi$ locally as
$$\varphi=\frac{1}{6}\varphi_{ijk}dx^{i}\wedge dx^{j}\wedge dx^{k},\ \psi=\frac{1}{24}\psi_{ijkl}dx^{i}\wedge dx^{j}\wedge dx^{k}\wedge dx^{l}.$$
The relation of $\varphi,\ \psi,\ g$ satisfy the following:
\begin{lemma}[\cite{Bryant 2006,Spiros introduce to G2}]\label{lemma2.1}
    In local coordinates on M, the tensors $\varphi,\ \psi$ and $g$ satisfy the following relations:
    \begin{align}
        \varphi_{ijk}\varphi_{abc}g^{kc}&=g_{ia}g_{jb}-g_{ib}g_{ja}+\psi_{ijab},\\
        \varphi_{ijk}\varphi_{abc}g^{jb}g^{kc}&=6g_{ia},\\
        \psi_{ijkl}\psi_{abcd}g^{jb}g^{kc}g^{ld}&=24g_{ia},\\
        \varphi_{ijq}\psi_{abkl}g^{ia}g^{jb}&=4\varphi_{qkl}.
    \end{align}
\end{lemma}

\subsection{$G_{2}$-decomposition of $\Omega^{2}(M)$ and $\Omega^{3}(M)$}\label{subsection2.2} 

The group $G_{2}$ acts irreducibly on $\mathbb{R}^{7}$ (and hence on $\wedge^{1}(\mathbb{R}^{7})^{\ast}$ and $\wedge^{6}(\mathbb{R}^{7})^{\ast}$), but it
acts reducibly on $\wedge^{k}(\mathbb{R}^{7})^{\ast}$ for $2\leq k\leq 5$. Hence a $G_{2}$ structure $\varphi$ induces splittings
of the bundles $\wedge^{k}(T^{\ast}M)(2\leq k\leq5)$ into direct summands, which we denote by
$\wedge^{k}_{l}(T^{\ast}M,\varphi)$ with $l$ being the rank of the bundle. We let the space of {smooth} sections
of $\wedge^{k}_{l}(T^{\ast}M,\varphi)$ be $\Omega^{k}_{l}(M)$. Define the natural projections
$$\pi^{k}_{l}:\Omega^{k}(M)\longrightarrow \Omega^{k}_{l}(M),\ \ \alpha\longmapsto \pi^{k}_{l}(\alpha).$$
Then we have
\begin{align}
    \Omega^{2}(M)&=\Omega^{2}_{7}(M)\oplus\Omega^{2}_{14}(M),\notag\\
    \Omega^{3}(M)&=\Omega^{3}_{1}(M)\oplus\Omega^{3}_{7}(M)\oplus\Omega^{3}_{27}(M)\notag.
\end{align}
where each component is determined by
\begin{align}
    \Omega^{2}_{7}(M)&=\{X\lrcorner\varphi:X\in {\Gamma}(TM)\}=\{\beta\in\Omega^{2}(M):\ast_{\varphi}(\varphi\wedge\beta)=2\beta\},\notag\\
    \Omega^{2}_{14}(M)&=\{\beta\in\Omega^{2}(M):\psi\wedge\beta=0\}=\{\beta\in\Omega^{2}(M):\ast_{\varphi}(\varphi\wedge\beta)=-\beta\},\notag
\end{align}
and
\begin{align}
    \Omega^{3}_{1}(M)&=\{f\varphi:f\in C^{\infty}(M)\},\notag\\
    \Omega^{3}_{7}(M)&=\{\ast_{\varphi}(\varphi\wedge\alpha):\alpha\in\Omega^{1}(M)\}=\{X\lrcorner\psi:X\in {\Gamma}(TM)\},\notag\\
    \Omega^{3}_{27}(M)&=\{\eta\in\Omega^{3}(M):\eta\wedge\varphi=\eta\wedge\psi=0\}.\notag
\end{align}

\begin{remark}\label{remark2.2}
    $\Omega^{4}$ and $\Omega^{5}$ have the corresponding decompositions by Hodge duality. The more details for $G_{2}$-decomposition see \cite{Bryant 2006,Spiros introduce to G2}.
\end{remark}

\subsection{The torsion tensors of $G_{2}$-structure}
\label{subsection2.3}

By the definition  of $G_{2}$ decomposition, we can find unique differential forms
$\tau_{0}\in {C^{\infty}}(M),\tau_{1},\widetilde{\tau}_{1}\in\Omega^{1}(M),\tau_{2}\in\Omega^{2}_{14}(M)$ and $\tau_{3}\in\Omega^{3}_{27}(M)$ such that (see \cite{Bryant 2006})
\begin{align}
    d\varphi&=\tau_{0}\psi+3\!\ \tau_{1}\wedge\varphi+\ast_{\varphi}\tau_{3},\\
    d\psi&=4\!\ \widetilde{\tau}_{1}\wedge\psi+\tau_{2}\wedge\varphi.
\end{align}
In fact, {Bryant \cite{Bryant 2006}} proved that $\tau_{1}=\widetilde{\tau}_{1}$. We call $\tau_{0}$ the \textit{scalar torsion}, $\tau_{1}$ the \textit{vector torsion}, $\tau_{2}$ the \textit{Lie algebra torsion}, and $\tau_{3}$ the \textit{symmetric traceless torsion}. We also call $\tau_{\varphi}:=\{\tau_{0},\tau_{1},\tau_{2},\tau_{3}\}$ the intrinsic torsion forms of the $G_{2}$-structure $\varphi$.
From \cite{Spiros flow of G2,Spiros introduce to G2}, the full torsion tensor $\mathbf{T}=\mathbf{T}_{ij}dx^{i}\otimes dx^{j}$ satisfies the followings:
\begin{align}\label{2.7}
    \nabla_{i}\varphi_{jkl}&=\mathbf{T}_{i}^{\ m}\psi_{mjkl},\\
    \mathbf{T}_{i}^{\ j}&=\frac{1}{24}\nabla_{i}\varphi_{lmn}\psi^{jlmn},
\end{align}
and 
\begin{align}
    \nabla_{m}\psi_{ijkl}=-(\mathbf{T}_{mi}\varphi_{jkl}-\mathbf{T}_{mj}\varphi_{ikl}-\mathbf{T}_{mk}\varphi_{jil}-\mathbf{T}_{ml}\varphi_{jki}).
\end{align}
The full torsion tensor $\mathbf{T}_{ij}$ is related to the
intrinsic torsion forms by the following:
\begin{align}
\label{Def of T}\mathbf{T}_{ij}=\frac{\tau_{0}}{4}g_{ij}-(\tau_{3})_{ij}-(\tau_{1}^{\#}\lrcorner\varphi)_{ij}-\frac{1}{2}(\tau_{2})_{ij}
\end{align}
or as $2$-tensors,
$$\mathbf{T}=\frac{\tau_{0}}{4}g_{\varphi}-\tau_{3}-\tau_{1}^{\#}\lrcorner\varphi-\frac{1}{2}\tau_{2},$$
where $(\tau_{1}^{\#}\lrcorner\varphi)_{ij}=(\tau_{1}^{\#})^{l}\varphi_{lij}$ and $\#$ is the isomorphism from $1$-form to vector fields.

${}$

If $\varphi$ is closed, which means $d\varphi=0$, then $\tau_{0},\tau_{1},\tau_{3}$ are all zero from (2.5), so the only nonzero torsion form is 
$$\tau\equiv\tau_{2}=\frac{1}{2}(\tau_{2})_{ij}dx^{i}{\wedge} dx^{j}=\frac{1}{2}\tau_{ij}dx^{i}{\wedge} dx^{j}.$$
Then according to $\eqref{Def of T}$, we have 
$$
\mathbf{T}_{ij}=-\mathbf{T}_{ji}=-\frac{1}{2}(\tau_{2})_{ij}\ \ \text{or equivalently}\ \ \mathbf{T}=-\frac{1}{2}\tau,
$$
so that $\mathbf{T}$ is a skew-symmetric 2-tensor, i.e., a $2$-form. Since $d\psi=\tau\wedge\varphi=-\ast_{\varphi}\tau$, we get $d_{\varphi}^{\ast}\tau=\ast_{\varphi}d\ast_{\varphi}\tau=-\ast_{\varphi}d^{2}\psi=0$ which is given in local coordinates by
\begin{align}
\nabla^{i}\tau_{ij}=0.
\end{align}

\subsection{Curvature tensor and torsion}\label{subsection2.4}
Since $\varphi$ determines the unique metric $g\equiv g_{\varphi}$ on $M$, we obtain the Levi-Civita connection $\nabla\equiv\nabla_{g}$ induced by $g$, and also the Hodge Laplacian
$$\Delta_{\varphi}=dd^{\ast}_{\varphi}+d^{\ast}_{\varphi}d$$
of $g$. 
The Riemann curvature $(3,1)$-tensor field ${\rm Rm}$ of $g$ is defined by
$$
{\rm Rm}(X,Y)Z:=\nabla_{X}\nabla_{Y}Z-\nabla_{Y}\nabla_{X}Z-\nabla_{[X,Y]}Z,
$$
where $X,Y,Z$ are vector fields on $M$. In local coordinates $\{x^{1},\cdots,x^{7}\}$ on $M$, if we choose $\displaystyle{X=\frac{\partial}{\partial x^{i}}, Y=\frac{\partial}{\partial x^{j}}, Z=\frac{\partial}{\partial x^{k}}}$, $1\leq i,j,k\leq 7$, then we can write
$$
{\rm Rm}\left(\frac{\partial}{\partial x^{i}},\frac{\partial}{\partial x^{j}}\right)\frac{\partial}{\partial x^{k}}=R_{ijk}^{\ \ \ l}\frac{\partial}{\partial x^{l}},\ \ R_{ijkl}=g_{{lm}}R_{ijk}^{\ \ \ m}.
$$
We have the Ricci identities when we commute covariant derivatives
of a $(0,k)$-tensor $\alpha$:
\begin{align}\label{2.12}
    (\nabla_{i}\nabla_{j}-\nabla_{j}\nabla_{i})\alpha_{i_{1}i_{2}\cdots i_{k}}=-\sum_{l=1}^{k}R_{iji_{l}}^{\ \ \ m}\alpha_{i_{1}\cdots i_{l-1}m i_{l+1}\cdots  i_{k}}.
\end{align}
For the torsion form, We have the following Bianchi-type identity.

\begin{lemma}\label{lemma2.3}
The full torsion tensor $\mathbf{T}$ satisfies the following identity: 
$$\nabla_{i}\mathbf{T}_{jk}-\nabla_{j}\mathbf{T}_{ik}=-\left(\frac{1}{2}R_{ijmn}+\mathbf{T}_{im}\mathbf{T}_{jn}\right)\varphi_{k}^{\ mn}.$$
\end{lemma}

\begin{proof}
    For the details see \cite{Spiros flow of G2} and \cite{Lotay-Wei Shi-estimate}.
\end{proof}

For a closed $G_{2}$-structure $\varphi$, the Ricci curvature is given by (see \cite{Lotay-Wei Shi-estimate})
\begin{align}
    R_{jk}=g^{il}R_{ijkl}=-(\nabla_{i}\mathbf{T}_{jm})\varphi_{k}^{\ im}-\mathbf{T}_{j}^{\ i}\mathbf{T}_{ik}.
\end{align}
Moreover, the factor $\nabla_{i}\mathbf{T}_{jm}$ can be expressed as

\begin{lemma}\label{lemma2.4}
    For any closed $G_{2}$-structure $\varphi$, we have
\begin{align}
    \nabla_{i}\mathbf{T}_{jk}&=-\frac{1}{4}R_{ijmn}\varphi_{k}^{\ mn}-\frac{1}{4}R_{kjmn}\varphi_{i}^{\ mn}+\frac{1}{4}R_{ikmn}\varphi_{j}^{\ mn}\notag\\
    &\quad -\frac{1}{2}\mathbf{T}_{im}\mathbf{T}_{jn}\varphi_{k}^{\ mn}-\frac{1}{2}\mathbf{T}_{km}\mathbf{T}_{jn}\varphi_{i}^{\ mn}+\frac{1}{2}\mathbf{T}_{im}\mathbf{T}_{kn}\varphi_{j}^{\ mn}\notag.
\end{align}
\end{lemma}

\begin{proof}
    This lemma has been proved in \cite{Lotay-Wei Shi-estimate}.
\end{proof}

We define the inner product for any two $k$-tensors $A=A_{i_{1}\cdots i_{k}}dx^{i_{1}}\otimes\cdots\otimes dx^{i_{k}},B=B_{i_{1}\cdots i_{k}}dx^{i_{1}}\otimes\cdots\otimes dx^{i_{k}}$ as
$$
\langle A,B\rangle=A_{i_{1}\cdots i_{k}}B^{i_{1}\cdots i_{k}},
$$
then the norm of any $k$-tensor $A=A_{i_{1}\cdots i_{k}}dx^{i_{1}}\otimes\cdots\otimes dx^{i_{k}}$ is
$$
|A|^{2}=\langle A,A\rangle=A_{i_{1}\cdots i_{k}}A^{i_{1}\cdots i_{k}},
$$
and the $C^{1}$ norm of any $k$-tensor $A$ on 7-dimensional manifold denote by
\begin{align}
\label{C^{1} norm}|A|_{C^{1}}:=\sqrt{|A|^{2}+|\nabla A|^{2}}.
\end{align}
The scalar curvature $R$ of a closed $G_{2}$-structure $\varphi$ satisfies
\begin{align}
    R=g^{ij}R_{ij}=-|\mathbf{T}|^{2}=-g^{ik}g^{jl}\mathbf{T}_{ij}\mathbf{T}_{kl}.
\end{align}
The Einstein tensor or traceless Ricci tensor $E$ is defined as
$$
E_{ij}=R_{ij}-\frac{R}{7}g_{ij},
$$
Then the Riemann curvature tensor can be decomposed in the following way
$$
{\rm Rm}=\frac{R}{84}g\circ g+\frac{1}{5}E\circ g+W,
$$
where $W$ is the Weyl tensor and ``$\circ$" is the Kulkarni-Nomizu product defined by 
$$
(\mathbf{\alpha}\circ {\bf\beta})_{ijkl}:={\bf\alpha}_{il}{\bf\beta}_{jk}+{\bf\alpha}_{jk}{\bf\beta}_{il}-{\bf\alpha}_{ik}{\bf\beta}_{jl}-{\bf\alpha}_{jl}{\bf\beta}_{ik}
$$
for any symmetric $2$-tensors ${\bf\alpha}$ and ${\bf\beta}$. In local coordinates, the Weyl tensor is given by 
\begin{align}
    \label{weyl tensor}W_{ijkl}&=R_{ijkl}-\frac{1}{5}(g_{il}R_{jk}+g_{jk}R_{il}-g_{ik}R_{jl}-g_{jl}R_{ik})\\
    &\quad+\frac{1}{30}R(g_{il}g_{jk}-g_{ik}g_{jl}).\notag
\end{align}

For convenience, we often use $g$, $\Delta$, ${\rm Rm}$, $ W$, ${\rm Ric}$, $R$, $\nabla$, $|\cdot|$, $\varphi$, $\mathbf{T}$, $\psi$ to replace $g(t)$, $\Delta_{\varphi(t)}$, ${\rm Rm}(g(t))$, $W(g(t))$, ${\rm Ric}(g(t))$, $R(g(t))$, $\nabla_{g(t)}$, $|\cdot|_{g(t)}$, $\varphi(t)$, $\mathbf{T}_{\varphi(t)}$, $\psi_{\varphi(t)}$. We always omit
the time variable $t$.

\section{Evolution equations under the Laplacian $G_{2}$ flow}\label{section3}

In this section, we will calculate some evolution equations under the Laplacian $G_{2}$ flow $\eqref{The closed Laplacian flow}$.
From \cite{Lotay-Wei Shi-estimate}, we see that the associated metric tensor $g(t)$ evolves by
\begin{align}
\partial_{t}g_{ij}=2h_{ij},
\end{align}
where
$$
h_{ij}=-R_{ij}-\frac{1}{3}|\mathbf{T}|^{2}g_{ij}-2\widehat{\mathbf{T}}_{ij}
$$
and $\widehat{\mathbf{T}}_{ij}=\mathbf{T}_{i}^{\ k}\mathbf{T}_{kj}$.
We denote by
$$
S_{ij}:=R_{ij}+\frac{1}{3}|\mathbf{T}|^{2}g_{ij}+2\widehat{\mathbf{T}}_{ij}=-h_{ij},
$$
the components of the symmetric $2$-tensor ${\rm Sic}(g(t))$. Then the equation (3.1) can be written as 
$$
\partial_{t}g_{ij}=-2S_{ij},
$$
and the trace of ${\rm Sic}(g(t))$ is
$$
S(g(t))={\rm tr}_{g(t)}({\rm Sic}(g(t))=R(g(t))+\frac{7}{3}|\mathbf{T}(t)|^{2}_{g(t)}-2|\mathbf{T}(t)|^{2}_{g(t)}=\frac{2}{3}R(g(t)).
$$

For a geometric flow 
\begin{align}
\partial_{t}g_{ij}=\eta_{ij},
\label{geometric flow}
\end{align}
where $\eta_{ij}$ is a family of symmetric $2$-tensor. From formula (2.30) and (2.31) in \cite{C-L-N Ricci flow}, we have the following evolution equations for curvature:

\begin{lemma}\label{lemma3.1}
    Under the flow $\eqref{geometric flow}$, we have
\begin{align}
    \partial_{t}R_{ij}&=-\frac{1}{2}\left(\Delta_{L}\eta_{ij}+\nabla_{i}\nabla_{j}{\rm tr}_{g(t)}(\eta)-\nabla_{i}({\rm div}_{g(t)}\eta)_{j}-\nabla_{j}({\rm div}_{g(t)}\eta)_{i}\right),\\
    \partial_{t}R(g(t))&=-\Delta_{g(t)}{\rm tr}_{g(t)}(\eta(t))+{\rm div}_{g(t)}\left({\rm div}_{g(t)}\eta(t)\right)-\left\langle{\rm Ric}(g(t)),\eta(t)\right\rangle_{g(t)}.
\end{align}
where $\Delta_{L}$ denotes the Lichnerowicz Laplacian
$$\Delta_{L}\eta_{ij}=\Delta\eta_{ij}-R_{i}^{\ k}\eta_{kj}-R_{j}^{\ k}\eta_{ki}+2R_{pijl}\eta^{pl} $$
and ${\rm tr}_{g(t)}(\eta(t))=g^{ij}\eta_{ij}$, $({\rm div}_{g(t)}\eta)_{j}=\nabla^{m}\eta_{mj}$.
\end{lemma}

Next, we calculate the evolution equations of Ricci curvature and scalar curvature under the Laplacian $G_{2}$ flow $\eqref{The closed Laplacian flow}$.

\begin{lemma}\label{lemma3.2}
    Under the Laplacian $G_{2}$ flow $\eqref{The closed Laplacian flow}$, we have
    \begin{align}
        \partial_{t}R_{ij}&=\Delta S_{ij}-2R_{i}^{\ p}R_{pj}-2R_{i}^{\ p}\widehat{\mathbf{T}}_{pj}-2R_{j}^{\ p}\widehat{\mathbf{T}}_{pi}+2R_{pijl}R^{pl}+4R_{pijl}\widehat{\mathbf{T}}^{pl}\notag\\
        &\quad-\frac{1}{3}\nabla_{i}\nabla_{j}|\mathbf{T}|^{2}-2\nabla_{i}\nabla^{p}\widehat{\mathbf{T}}_{pj}-2\nabla_{j}\nabla^{p}\widehat{\mathbf{T}}_{pi}\notag,
    \end{align}
    where $\widehat{\mathbf{T}}_{pj}=\mathbf{T}_{p}^{\ m}\mathbf{T}_{mj}$.
\end{lemma}

\begin{proof}
    From the definition of ${\rm Sic}(g(t))$ and Lemma \ref{lemma3.1}, we have
    \begin{align}
        \partial_{t}R_{ij}&=-\frac{1}{2}\left(-\Delta_{L}2S_{ij}-2\nabla_{i}\nabla_{j}S+2\nabla_{i}({\rm div}_{g(t)} {\rm Sic})_{j}+2\nabla_{j}({\rm div}_{g(t)}{\rm Sic})_{i}\right)\notag\\
        &=\Delta_{L}S_{ij}+\frac{2}{3}\nabla_{i}\nabla_{j}R-\nabla_{i}\nabla^{p}R_{pj}-\nabla_{j}\nabla^{p}R_{pi}-\frac{2}{3}\nabla_{i}\nabla_{j}|\mathbf{T}|^{2}\notag\\
        &\quad-2\nabla_{i}\nabla^{p}\widehat{\mathbf{T}}_{pj}-2\nabla_{j}\nabla^{p}\widehat{\mathbf{T}}_{pi}\notag\\
        &=\Delta_{L}S_{ij}-\frac{1}{3}\nabla_{i}\nabla_{j}|\mathbf{T}|^{2}-2\nabla_{i}\nabla^{p}\widehat{\mathbf{T}}_{pj}-2\nabla_{j}\nabla^{p}\widehat{\mathbf{T}}_{pi}\notag,
    \end{align}
    where we use $\nabla^{p}R_{pj}=\frac{1}{2}\nabla_{j}R$, note that
    \begin{align}
        \Delta_{L}S_{ij}&=\Delta S_{ij}-R_{i}^{\ p}S_{pj}-R_{j}^{\ p}S_{pi}+2R_{pijl}S^{pl}\notag\\
        &=\Delta R_{ij}+\frac{1}{3}g_{ij}\Delta|\mathbf{T}|^{2}+2\Delta\widehat{\mathbf{T}}_{ij}-R_{i}^{\ p}\left(R_{pj}+\frac{1}{3}|\mathbf{T}|^{2}g_{pj}+2\widehat{\mathbf{T}}_{pj}\right)\notag\\
        &\quad-R_{j}^{\ p}\left(R_{ip}+\frac{1}{3}|\mathbf{T}|^{2}g_{ip}+2\widehat{\mathbf{T}}_{ip}\right)+2R_{pijl}\left(R^{pl}+\frac{1}{3}|\mathbf{T}|^{2}g^{pl}+2\widehat{\mathbf{T}}^{pl}\right)\notag\\
        &=\Delta R_{ij}+\frac{1}{3}g_{ij}\Delta|\mathbf{T}|^{2}+2\Delta\widehat{\mathbf{T}}_{ij}-R_{i}^{\ p}R_{pj}-\frac{1}{3}|\mathbf{T}|^{2}R_{ij}-2R_{i}^{\ p}\widehat{\mathbf{T}}_{pj}\notag\\
        &\quad-R_{j}^{\ p}R_{ip}-\frac{1}{3}|\mathbf{T}|^{2}R_{ji}-2R_{j}^{\ p}\widehat{\mathbf{T}}_{pi}+2R_{pijl}R^{pl}+\frac{2}{3}|\mathbf{T}|^{2}R_{ij}\notag\\
        &\quad+4R_{pijl}\widehat{\mathbf{T}}^{pl}\notag\\
        &=\Delta R_{ij}+\frac{1}{3}g_{ij}\Delta|\mathbf{T}|^{2}+2\Delta\widehat{\mathbf{T}}_{ij}-2R_{i}^{\ p}R_{pj}-2R_{i}^{\ p}\widehat{\mathbf{T}}_{pj}\notag\\
        &\quad-2R_{j}^{\ p}\widehat{\mathbf{T}}_{pi}+2R_{pijl}R^{pl}+4R_{pijl}\widehat{\mathbf{T}}^{pl}\notag.
    \end{align}
    Together with the above, we get the evolution equation of Ricci curvature.
\end{proof}

\begin{lemma}\label{lemma3.3}
Under the Laplacian $G_{2}$ flow $\eqref{The closed Laplacian flow}$, the evolution equation for the norm of Ricci curvature is the following:
    \begin{align}
    \partial_{t}|{\rm Ric}|^{2}&=\Delta|{\rm Ric}|^{2}-2|\nabla{\rm Ric}|^{2}+4R_{pijl}R^{pl}R^{ij}+\frac{4}{3}|\mathbf{T}|^{2}|{\rm Ric}|^{2}+8R_{pijl}\widehat{\mathbf{T}}^{pl}R^{ij}\notag\\
    &\quad+\frac{2}{3}R\Delta|\mathbf{T}|^{2}+4R^{ij}\Delta\widehat{\mathbf{T}}_{ij}-\frac{2}{3}R^{ij}\nabla_{i}\nabla_{j}|\mathbf{T}|^{2}-8R^{ij}\nabla_{i}\nabla^{p}\widehat{\mathbf{T}}_{pj}\notag.
    \end{align}
\end{lemma}

\begin{proof}
    According to Lemma \ref{lemma3.2} and by calculating yields
    \begin{align}
        \partial_{t}|{\rm Ric}|^{2}&=\partial_{t}\left(g^{ik}g^{jl}R_{ij}R_{kl}\right)=2g^{ik}\left(\partial_{t}g^{jl}\right)R_{ij}R_{kl}+2R^{ij}\partial_{t}R_{ij}\notag\\
        &=4g^{ik}S^{jl}R_{ij}R_{kl}+2R^{ij}\Big{(}\Delta S_{ij}-2R_{i}^{\ p}R_{pj}-4R_{i}^{\ p}\widehat{\mathbf{T}}_{pj}+2R_{pijl}R^{pl}\notag\\
        &\quad+4R_{pijl}\widehat{\mathbf{T}}^{pl}-\frac{1}{3}\nabla_{i}\nabla_{j}|\mathbf{T}|^{2}-4\nabla_{i}\nabla^{p}\widehat{\mathbf{T}}_{pj}\Big{)}\notag\\
        &=4R_{j}^{\ k}R_{kl}R^{lj}+\frac{4}{3}|\mathbf{T}|^{2}|{\rm Ric}|^{2}+8R_{j}^{\ k}R_{kl}\widehat{\mathbf{T}}^{jl}+2R^{ij}\Delta R_{ij}+\frac{2}{3}R\Delta|\mathbf{T}|^{2}\notag\\
        &\quad+4R^{ij}\Delta\widehat{\mathbf{T}}_{ij}-4R^{ij}R_{jp}R_{\ i}^{p}-8R^{ij}R_{i}^{\ p}\widehat{\mathbf{T}}_{pj}+4R_{pijl}R^{ij}R^{pl}\notag\\
        &\quad+8R_{pijl}R^{ij}\widehat{\mathbf{T}}^{pl}-\frac{2}{3}R^{ij}\nabla_{i}\nabla_{j}|\mathbf{T}|^{2}-8R^{ij}\nabla_{i}\nabla^{p}\widehat{\mathbf{T}}_{pj}\notag\\
        &=\Delta|{\rm Ric}|^{2}-2|\nabla{\rm Ric}|^{2}+4R_{pijl}R^{pl}R^{ij}+\frac{4}{3}|\mathbf{T}|^{2}|{\rm Ric}|^{2}+8R_{pijl}\widehat{\mathbf{T}}^{pl}R^{ij}\notag\\
        &\quad+\frac{2}{3}R\Delta|\mathbf{T}|^{2}+4R^{ij}\Delta\widehat{\mathbf{T}}_{ij}-\frac{2}{3}R^{ij}\nabla_{i}\nabla_{j}|\mathbf{T}|^{2}-8R^{ij}\nabla_{i}\nabla^{p}\widehat{\mathbf{T}}_{pj}\notag.
    \end{align}
    where the last equality {uses} $\Delta|{\rm Ric}|^{2}=2R^{ij}\Delta R_{ij}+2|\nabla{\rm Ric}|^{2}$.
\end{proof}

\begin{lemma}\label{lemma3.4}
    Under the closed Laplacian $G_{2}$ flow $\eqref{The closed Laplacian flow}$, we have
    \begin{align}
        \partial_{t}R=\Delta R+2|{\rm Ric}|^{2}-\frac{2}{3}R^{2}+4R_{ijmn}\mathbf{T}^{in}\mathbf{T}^{mj}-4\nabla^{j}\mathbf{T}_{im}\cdot\nabla^{i}\mathbf{T}^{m}_{\ \ j}.
    \end{align}
\end{lemma}

\begin{proof}
    From Lemma \ref{lemma3.2}, we have
    \begin{align}\label{3.6}
        \partial_{t}R&=\partial_{t}\left(g^{ij}R_{ij}\right)=\left(\partial_{t}g^{ij}\right)R_{ij}+g^{ij}\partial_{t}R_{ij}\\
        &=2S^{ij}R_{ij}+g^{ij}\Big{(}\Delta S_{ij}-2R_{i}^{\ p}R_{pj}-2R_{i}^{\ p}\widehat{\mathbf{T}}_{pj}-2R_{j}^{\ p}\widehat{\mathbf{T}}_{pi}+2R_{pijl}R^{pl}\notag\\
        &\quad+4R_{pijl}\widehat{\mathbf{T}}^{pl}-\frac{1}{3}\nabla_{i}\nabla_{j}|\mathbf{T}|^{2}-2\nabla_{i}\nabla^{p}\widehat{\mathbf{T}}_{pj}-2\nabla_{j}\nabla^{p}\widehat{\mathbf{T}}_{pi}\Big{)}\notag\\
        &=\Delta R+2|{\rm Ric}|^{2}-\frac{2}{3}R^{2}-4\nabla^{i}\nabla^{j}\widehat{\mathbf{T}}_{ij}+4R^{ij}\widehat{\mathbf{T}}_{ij}\notag.
    \end{align}
    Note that
    \begin{align}
        \nabla^{i}\nabla^{j}\widehat{\mathbf{T}}_{ij}&=\nabla^{i}\nabla^{j}\left(\mathbf{T}_{im}\mathbf{T}^{m}_{\ \ j}\right)=\nabla^{i}\left(\nabla^{j}\mathbf{T}_{im}\cdot \mathbf{T}^{m}_{\ \ j}+\mathbf{T}_{im}\cdot\nabla^{j}\mathbf{T}^{m}_{\ \ j}\right)\\
        &=\nabla^{i}\nabla^{j}\mathbf{T}_{im}\cdot \mathbf{T}^{m}_{\ \ j}+\nabla^{j}\mathbf{T}_{im}\cdot\nabla^{i}\mathbf{T}^{m}_{\ \ j}\notag\\
        &=\left(\nabla^{j}\nabla^{i}\mathbf{T}_{im}-R^{ij\ p}_{\ \ i}\mathbf{T}_{pm}-R^{ij\ \ p}_{\ \ m}\mathbf{T}_{ip}\right)\mathbf{T}^{m}_{\ \ j}+\nabla^{j}\mathbf{T}_{im}\cdot\nabla^{i}\mathbf{T}^{m}_{\ \ j}\notag\\ 
        &=R^{jp}\widehat{\mathbf{T}}_{pj}-R_{ijmp}\mathbf{T}^{ip}\mathbf{T}^{mj}+\nabla^{j}\mathbf{T}_{im}\cdot\nabla^{i}\mathbf{T}^{m}_{\ \ j}\notag,
    \end{align}
    where we use $\nabla^{i}\mathbf{T}_{im}=0$ and (2.12). 
    Together with the above, we prove this lemma.
\end{proof}

In the sequel, we define
$$
\widetilde{\rm Ric}:={\rm Ric}+\frac{c}{7}g,\  \widetilde{R}:={\rm tr}\left(\widetilde{\rm Ric}\right)=R+c,
$$
where $c$ is any constant. Then
$$
|\widetilde{\rm Ric}|^{2}=|{\rm Ric}|^{2}+\frac{c^{2}}{7}+\frac{2c}{7}R.
$$
By calculating, we can get the following corollary:

\begin{corollary}\label{corollary3.5}
    Under the Laplacian $G_{2}$ flow $\eqref{The closed Laplacian flow}$, we have
    \begin{align}
        \partial_{t}|\widetilde{\rm Ric}|^{2}=\Delta |\widetilde{\rm Ric}|^{2}-2|\nabla\widetilde{\rm Ric}|^{2}+4R_{pijl}\widetilde{R}^{pl}\widetilde{R}^{ij}+I+J,
    \end{align}
    where $I,J$ are defined in $\eqref{I and J}$.
\end{corollary}

\begin{proof}
According to Lemma \ref{lemma3.3} and $\eqref{3.6}$, we have
    \begin{align}
    \partial_{t}|\widetilde{\rm Ric}|^{2}&=\partial_{t}|{\rm Ric}|^{2}+\frac{2c}{7}\partial_{t}R\notag\\
    &=\Delta|{\rm Ric}|^{2}-2|\nabla{\rm Ric}|^{2}+4R_{pijl}R^{pl}R^{ij}+\frac{4}{3}|\mathbf{T}|^{2}|{\rm Ric}|^{2}+8R_{pijl}\widehat{\mathbf{T}}^{pl}R^{ij}\notag\\
    &\quad+\frac{2}{3}R\Delta|\mathbf{T}|^{2}+4R^{ij}\Delta\widehat{\mathbf{T}}_{ij}-\frac{2}{3}R^{ij}\nabla_{i}\nabla_{j}|\mathbf{T}|^{2}-8R^{ij}\nabla_{i}\nabla^{p}\widehat{\mathbf{T}}_{pj}\notag\\
    &\quad+\frac{2c}{7}\left(\Delta R+2|{\rm Ric}|^{2}-\frac{2}{3}R^{2}-4\nabla^{i}\nabla^{j}\widehat{\mathbf{T}}_{ij}+4R^{ij}\widehat{\mathbf{T}}_{ij}\right)\notag.
    \end{align}
    Note that $R=\widetilde{R}-c$ and
    $$
    |{\rm Ric}|^{2}=|\widetilde{\rm Ric}|^{2}+\frac{c^{2}}{7}-\frac{2c}{7}\widetilde{R},
    $$
    Through computing, we have
    \begin{align}
        &\quad\Delta |\widetilde{\rm Ric}|^{2}-2|\nabla\widetilde{\rm Ric}|^{2}=\Delta|{\rm Ric}|^{2}+\frac{2c}{7}\Delta R-2|\nabla{\rm Ric}|^{2},\notag\\
        &\quad4R_{pijl}R^{pl}R^{ij}+\frac{4}{3}|\mathbf{T}|^{2}|{\rm Ric}|^{2}+8R_{pijl}\widehat{\mathbf{T}}^{pl}R^{ij}\notag\\
        &=4R_{pijl}\left(\widetilde{R}^{pl}-\frac{c}{7}g^{pl}\right)\left(\widetilde{R}^{ij}-\frac{c}{7}g^{ij}\right)-\frac{4}{3}(\widetilde{R}-c)\left(|\widetilde{\rm Ric}|^{2}+\frac{c^{2}}{7}-\frac{2c}{7}\widetilde{R}\right)\notag\\
        &\quad+8R_{pijl}\widehat{\mathbf{T}}^{pl}\left(\widetilde{R}^{ij}-\frac{c}{7}g^{ij}\right)\notag\\
        &=4R_{pijl}\widetilde{R}^{pl}\widetilde{R}^{ij}-\frac{8c}{7}\widetilde{R}^{ij}\left(\widetilde{R}_{ij}-\frac{c}{7}g_{ij}\right)+\frac{4}{49}c^{2}\left(\widetilde{R}-c\right)-\frac{4}{3}\widetilde{R}|\widetilde{\rm Ric}|^{2}\notag\\
        &\quad-\frac{4}{21}c^{2}\widetilde{R}+\frac{8}{21}c\widetilde{R}^{2}+\frac{4}{3}c|\widetilde{\rm Ric}|^{2}+\frac{4}{21}c^{3}-\frac{8}{21}c^{2}\widetilde{R}+8R_{pijl}\widehat{\mathbf{T}}^{pl}\widetilde{R}^{ij}\notag\\
        &\quad-\frac{8c}{7}\widehat{\mathbf{T}}^{pl}\left(\widetilde{R}_{pl}-\frac{c}{7}g_{pl}\right)\notag\\
        &=4R_{pijl}\widetilde{R}^{pl}\widetilde{R}^{ij}-\frac{4}{3}\widetilde{R}|\widetilde{\rm Ric}|^{2}+\frac{4}{21}c|\widetilde{\rm Ric}|^{2}+8R_{pijl}\widehat{\mathbf{T}}^{pl}\widetilde{R}^{ij}-\frac{8c}{7}\widehat{\mathbf{T}}^{pl}\widetilde{R}_{pl}\notag\\
        &\quad+\frac{8c}{21}\widetilde{R}^{2}-\frac{8}{49}c^{2}\widetilde{R}-\frac{8}{147}c^{3}\notag.
    \end{align}
    and
    \begin{align}
        &\quad\frac{2}{3}R\Delta|\mathbf{T}|^{2}+4R^{ij}\Delta\widehat{\mathbf{T}}_{ij}-\frac{2}{3}R^{ij}\nabla_{i}\nabla_{j}|\mathbf{T}|^{2}-8R^{ij}\nabla_{i}\nabla^{p}\widehat{\mathbf{T}}_{pj}\notag\\
        &=\frac{2}{3}\left(\widetilde{R}-c\right)\Delta|\mathbf{T}|^{2}+4\left(\widetilde{R}^{ij}-\frac{c}{7}g^{ij}\right)\Delta\widehat{\mathbf{T}}_{ij}-\frac{2}{3}\left(\widetilde{R}^{ij}-\frac{c}{7}g^{ij}\right)\nabla_{i}\nabla_{j}|\mathbf{T}|^{2}\notag\\
        &\quad-8\left(\widetilde{R}^{ij}-\frac{c}{7}g^{ij}\right)\nabla_{i}\nabla^{p}\widehat{\mathbf{T}}_{pj}\notag\\
        &=\frac{2}{3}\widetilde{R}\Delta|\mathbf{T}|^{2}-\frac{2}{3}c\Delta|\mathbf{T}|^{2}+4\widetilde{R}^{ij}\Delta\widehat{\mathbf{T}}_{ij}+\frac{4c}{7}\Delta|\mathbf{T}|^{2}-\frac{2}{3}\widetilde{R}^{ij}\nabla_{i}\nabla_{j}|\mathbf{T}|^{2}\notag\\
        &\quad+\frac{2c}{21}\Delta|\mathbf{T}|^{2}-8\widetilde{R}^{ij}\nabla_{i}\nabla^{p}\widehat{\mathbf{T}}_{pj}+\frac{8c}{7}\nabla^{i}\nabla^{j}\widehat{\mathbf{T}}_{ij}\notag\\
        &=\frac{2}{3}\widetilde{R}\Delta|\mathbf{T}|^{2}+4\widetilde{R}^{ij}\Delta\widehat{\mathbf{T}}_{ij}-\frac{2}{3}\widetilde{R}^{ij}\nabla_{i}\nabla_{j}|\mathbf{T}|^{2}-8\widetilde{R}^{ij}\nabla_{i}\nabla^{p}\widehat{\mathbf{T}}_{pj}+\frac{8c}{7}\nabla^{i}\nabla^{j}\widehat{\mathbf{T}}_{ij}\notag,
    \end{align}
    and
    \begin{align}
        &\quad\frac{2c}{7}\left(2|{\rm Ric}|^{2}-\frac{2}{3}R^{2}-4\nabla^{i}\nabla^{j}\widehat{\mathbf{T}}_{ij}+4R^{ij}\widehat{\mathbf{T}}_{ij}\right)\\
        &=\frac{4c}{7}\left(|\widetilde{\rm Ric}|^{2}+\frac{c^{2}}{7}-\frac{2c}{7}\widetilde{R}\right)-\frac{4c}{21}\left(\widetilde{R}-c\right)\left(\widetilde{R}-c\right)-\frac{8c}{7}\nabla^{i}\nabla^{j}\widehat{\mathbf{T}}_{ij}\notag\\
        &\quad+\frac{8c}{7}\left(\widetilde{R}^{ij}-\frac{c}{7}g^{ij}\right)\widehat{\mathbf{T}}_{ij}\notag\\
        &=\frac{4c}{7}|\widetilde{\rm Ric}|^{2}+\frac{4}{49}c^{3}-\frac{8}{49}c^{2}\widetilde{R}-\frac{4c}{21}\widetilde{R}^{2}+\frac{8}{21}c^{2}\widetilde{R}-\frac{4}{21}c^{3}\notag\\
        &\quad-\frac{8c}{7}\nabla^{i}\nabla^{j}\widehat{\mathbf{T}}_{ij}+\frac{8c}{7}\widetilde{R}^{ij}\widehat{\mathbf{T}}_{ij}-\frac{8}{49}c^{2}\widetilde{R}+\frac{8}{49}c^{3}\notag\\
        &=\frac{4c}{7}|\widetilde{\rm Ric}|^{2}-\frac{8c}{7}\nabla^{i}\nabla^{j}\widehat{\mathbf{T}}_{ij}+\frac{8c}{7}\widetilde{R}^{ij}\widehat{\mathbf{T}}_{ij}-\frac{4c}{21}\widetilde{R}^{2}+\frac{8}{147}c^{2}\widetilde{R}+\frac{8}{147}c^{3}\notag.
    \end{align}
    Together with the above, we get
    \begin{align}
        \partial_{t}|\widetilde{\rm Ric}|^{2}=\Delta |\widetilde{\rm Ric}|^{2}-2|\nabla\widetilde{\rm Ric}|^{2}+4R_{pijl}\widetilde{R}^{pl}\widetilde{R}^{ij}+I+J,
    \end{align}
    where
    \begin{align}
       \label{I and J} I&=-\frac{4}{3}\widetilde{R}|\widetilde{\rm Ric}|^{2}+\frac{16c}{21}|\widetilde{\rm Ric}|^{2}+8R_{pijl}\widehat{\mathbf{T}}^{pl}\widetilde{R}^{ij}+\frac{4c}{21}\widetilde{R}^{2}-\frac{16}{147}c^{2}\widetilde{R},\\
        J&=\frac{2}{3}\widetilde{R}\Delta|\mathbf{T}|^{2}+4\widetilde{R}^{ij}\Delta\widehat{\mathbf{T}}_{ij}-\frac{2}{3}\widetilde{R}^{ij}\nabla_{i}\nabla_{j}|\mathbf{T}|^{2}-8\widetilde{R}^{ij}\nabla_{i}\nabla^{p}\widehat{\mathbf{T}}_{pj}\notag.  
    \end{align}
    Thus we prove this corollary.
\end{proof}

\begin{corollary}\label{corollary3.6}
       Under the Laplacian $G_{2}$ flow $\eqref{The closed Laplacian flow}$, we have
       \begin{align}     \partial_{t}\widetilde{R}=\Delta\widetilde{R}+2|\widetilde{\rm Ric}|^{2}+H,
       \end{align}
       where $H$ is defined in $\eqref{H}$ .
\end{corollary}
\begin{proof}
    According to Lemma \ref{lemma3.4}, {by} calculating
    \begin{align}
        \partial_{t}\widetilde{R}&=\partial_{t}(R+c)=\partial_{t}R\\
        &=\Delta R+2|{\rm Ric}|^{2}-\frac{2}{3}R^{2}+4R_{ijmn}\mathbf{T}^{in}\mathbf{T}^{mj}-4\nabla^{j}\mathbf{T}_{im}\cdot\nabla^{i}\mathbf{T}^{m}_{\ \ j}\notag\\
        &=\Delta\widetilde{R}+2|\widetilde{\rm Ric}|^{2}+\frac{2}{7}c^{2}-\frac{4}{7}c\widetilde{R}-\frac{2}{3}\left(\widetilde{R}-c\right)\left(\widetilde{R}-c\right)\notag\\     &\quad+4R_{ijmn}\mathbf{T}^{in}\mathbf{T}^{mj}-4\nabla^{j}\mathbf{T}_{im}\cdot\nabla^{i}\mathbf{T}^{m}_{\ \ j}\notag\\
        &=\Delta\widetilde{R}+2|\widetilde{\rm Ric}|^{2}+\frac{2}{7}c^{2}-\frac{4}{7}c\widetilde{R}-\frac{2}{3}\widetilde{R}^{2}+\frac{4c}{3}\widetilde{R}-\frac{2}{3}c^{2}\notag\\
&\quad+4R_{ijmn}\mathbf{T}^{in}\mathbf{T}^{mj}-4\nabla^{j}\mathbf{T}_{im}\cdot\nabla^{i}\mathbf{T}^{m}_{\ \ j}\notag\\
        &=\Delta\widetilde{R}+2|\widetilde{\rm Ric}|^{2}+H,\notag
    \end{align}
    where 
    \begin{align}
        \label{H} H=-\frac{2}{3}\widetilde{R}^{2}+\frac{16}{21}c\widetilde{R}-\frac{8}{21}c^{2}+4R_{ijmn}\mathbf{T}^{in}\mathbf{T}^{mj}-4\nabla^{j}\mathbf{T}_{im}\cdot\nabla^{i}\mathbf{T}^{m}_{\ \ j}.
    \end{align}
    Then we get the desired result.
\end{proof}

\section{Curvature pinching estimate}\label{section4}

In this section, we study the curvature pinching estimate and apply this estimate to prove the main results.
For any positive number $\gamma$, we define {the pinching quantity}
$$f:=\frac{|E|^{2}}{\widetilde{R}^{\gamma}}=\frac{|\widetilde{\rm Ric}|^{2}}{\widetilde{R}^{\gamma}}-\frac{1}{7}\widetilde{R}^{2-\gamma}.$$

{
\begin{remark}
    Bryant \cite{Bryant 2006} studied the pinching of Ricci curvature for closed $G_{2}$-structure and obtained an upper bound. However, our denominator in $f$ is modified scalar curvature $\tilde{R}$, which is different from scalar curvature in \cite{Bryant 2006}.
\end{remark}}

{We get the pinching quantity} $f$ satisfies the following evolution equations.

\begin{lemma}\label{lemma4.1}
   Under the Laplacian $G_{2}$ flow $\eqref{The closed Laplacian flow}$, we have
   \begin{align}
       \partial_{t}f&=\Delta f+\frac{2(\gamma-1)}{\widetilde{R}}\langle\nabla f,\nabla\widetilde{R}\rangle-\frac{2}{\widetilde{R}^{\gamma+2}}\left|\widetilde{R}\cdot\nabla\widetilde{\rm Ric}-\nabla\widetilde{R}\cdot\widetilde{\rm Ric}\right|^{2}\\
        &\quad-\frac{(2-\gamma)(\gamma-1)}{\widetilde{R}^{2}}|\nabla\widetilde{R}|^{2}f+\frac{2}{\widetilde{R}^{\gamma+1}}\bigg[-\gamma|E|^{4}+2\widetilde{R}\cdot W_{pijl}E^{pl}E^{ij}\notag\\
        &\quad-\frac{4}{5}\widetilde{R}E^{3}+\left(\frac{5}{21}-\frac{\gamma}{7}\right)\widetilde{R}^{2}|E|^{2}+\frac{c}{21}\widetilde{R}|E|^{2}-\frac{2c}{49}\widetilde{R}^{3}\bigg]\notag\\
        &\quad+\frac{1}{\widetilde{R}^{\gamma}}(I+J)-\frac{\gamma}{\widetilde{R}^{\gamma+1}}|\widetilde{\rm Ric}|^{2}H
        -\frac{(2-\gamma)}{7}\frac{H}{\widetilde{R}^{\gamma-1}}\notag.
   \end{align}
   where $E^{3}=E_{ij}E^{j}_{\ l}E^{li}$.
\end{lemma}
\begin{proof}
    From Corollary \ref{corollary3.5} and Corollary \ref{corollary3.6}, we have
    \begin{align}
        \partial_{t}\left(\frac{|\widetilde{\rm Ric}|^{2}}{\widetilde{R}^{\gamma}}\right)&=\frac{\partial_{t}|\widetilde{\rm Ric}|^{2}\cdot \widetilde{R}^{\gamma}-|\widetilde{\rm Ric}|^{2}\cdot\partial_{t}\widetilde{R}^{\gamma}}{\widetilde{R}^{2\gamma}}\\
        &=\frac{1}{\widetilde{R}^{\gamma}}\Delta|\widetilde{\rm Ric}|^{2}-\frac{2}{\widetilde{R}^{\gamma}}|\nabla\widetilde{\rm Ric}|^{2}+\frac{4}{\widetilde{R}^{\gamma}}R_{pijl}\widetilde{R}^{pl}\widetilde{R}^{ij}+\frac{1}{\widetilde{R}^{\gamma}}(I+J)\notag\\
        &\quad-\frac{\gamma}{\widetilde{R}^{\gamma+1}}|\widetilde{\rm Ric}|^{2}\Delta\widetilde{R}-\frac{2\gamma}{\widetilde{R}^{\gamma+1}}|\widetilde{\rm Ric}|^{4}-\frac{\gamma}{\widetilde{R}^{\gamma+1}}|\widetilde{\rm Ric}|^{2}H\notag\\
        &=\frac{1}{\widetilde{R}^{\gamma}}\Delta|\widetilde{\rm Ric}|^{2}-\frac{\gamma}{\widetilde{R}^{\gamma+1}}|\widetilde{\rm Ric}|^{2}\Delta\widetilde{R}-\frac{2}{\widetilde{R}^{\gamma}}|\nabla\widetilde{\rm Ric}|^{2}-\frac{2\gamma}{\widetilde{R}^{\gamma+1}}|\widetilde{\rm Ric}|^{4}\notag\\
        &\quad+\frac{4}{\widetilde{R}^{\gamma}}R_{pijl}\widetilde{R}^{pl}\widetilde{R}^{ij}+\frac{1}{\widetilde{R}^{\gamma}}(I+J)-\frac{\gamma}{\widetilde{R}^{\gamma+1}}|\widetilde{\rm Ric}|^{2}H\notag,
    \end{align}
    note that
    \begin{align}
        \nabla_{i}\left(\frac{|\widetilde{\rm Ric}|^{2}}{\widetilde{R}^{\gamma}}\right)&=\frac{\nabla_{i}|\widetilde{\rm Ric}|^{2}\cdot \widetilde{R}^{\gamma}-|\widetilde{\rm Ric}|^{2}\cdot\nabla_{i}\widetilde{R}^{\gamma}}{\widetilde{R}^{2\gamma}}=\frac{\nabla_{i}|\widetilde{\rm Ric}|^{2}}{\widetilde{R}^{\gamma}}-\frac{\gamma|\widetilde{\rm Ric}|^{2}}{\widetilde{R}^{\gamma+1}}\nabla_{i} \widetilde{R}\notag
    \end{align}
    and
    \begin{align}
        \Delta\left(\frac{|\widetilde{\rm Ric}|^{2}}{\widetilde{R}^{\gamma}}\right)&=\nabla_{i}\left(\frac{\nabla_{i}|\widetilde{\rm Ric}|^{2}}{\widetilde{R}^{\gamma}}-\frac{\gamma|\widetilde{\rm Ric}|^{2}}{\widetilde{R}^{\gamma+1}}\nabla_{i} \widetilde{R}\right)\\
        &=\frac{\Delta|\widetilde{\rm Ric}|^{2}\cdot\widetilde{R}^{\gamma}-\left\langle\nabla|\widetilde{\rm Ric}|^{2},\nabla\widetilde{R}^{\gamma}\right\rangle}{\widetilde{R}^{2\gamma}}-\frac{\gamma\left\langle\nabla|\widetilde{\rm Ric}|^{2},\nabla\widetilde{R}\right\rangle\cdot\widetilde{R}^{\gamma+1}}{\widetilde{R}^{2\gamma+2}}\notag\\
        &\quad-\frac{\gamma|\widetilde{\rm Ric}|^{2}\Delta\widetilde{R}\cdot\widetilde{R}^{\gamma+1}-\gamma|\widetilde{\rm Ric}|^{2}\left\langle\nabla R,\nabla \widetilde{R}^{\gamma+1}\right\rangle}{\widetilde{R}^{2\gamma+2}}\notag\\
        &=\frac{1}{\widetilde{R}^{\gamma}}\Delta|\widetilde{\rm Ric}|^{2}-\frac{\gamma}{\widetilde{R}^{\gamma+1}}|\widetilde{\rm Ric}|^{2}\Delta\widetilde{R}-\frac{2\gamma}{\widetilde{R}^{\gamma+1}}\left\langle\nabla|\widetilde{\rm Ric}|^{2},\nabla\widetilde{R}\right\rangle\notag\\
        &\quad+\frac{\gamma(\gamma+1)}{\widetilde{R}^{\gamma+2}}|\widetilde{\rm Ric}|^{2}|\nabla\widetilde{R}|^{2}\notag,
    \end{align}
    Thus we have 
    \begin{align}
        \partial_{t}\left(\frac{|\widetilde{\rm Ric}|^{2}}{\widetilde{R}^{\gamma}}\right)&=\Delta\left(\frac{|\widetilde{\rm Ric}|^{2}}{\widetilde{R}^{\gamma}}\right)+\frac{2\gamma}{\widetilde{R}^{\gamma+1}}\left\langle\nabla|\widetilde{\rm Ric}|^{2},\nabla\widetilde{R}\right\rangle-\frac{\gamma(\gamma+1)}{\widetilde{R}^{\gamma+2}}|\widetilde{\rm Ric}|^{2}|\nabla\widetilde{R}|^{2}\notag\\
        &\quad-\frac{2}{\widetilde{R}^{\gamma}}|\nabla\widetilde{\rm Ric}|^{2}+\frac{4}{\widetilde{R}^{\gamma}}R_{pijl}\widetilde{R}^{pl}\widetilde{R}^{ij}-\frac{2\gamma}{\widetilde{R}^{\gamma+1}}|\widetilde{\rm Ric}|^{4}\notag\\
        &\quad+\frac{1}{\widetilde{R}^{\gamma}}(I+J)-\frac{\gamma}{\widetilde{R}^{\gamma+1}}|\widetilde{\rm Ric}|^{2}H\notag\\
        &=\Delta\left(\frac{|\widetilde{\rm Ric}|^{2}}{\widetilde{R}^{\gamma}}\right)+\frac{2(\gamma-1)}{\widetilde{R}}\left\langle\nabla\left(\frac{|\widetilde{\rm Ric}|^{2}}{\widetilde{R}^{\gamma}}\right),\nabla\widetilde{R}\right\rangle+\frac{4}{\widetilde{R}^{\gamma}}R_{pijl}\widetilde{R}^{pl}\widetilde{R}^{ij}\notag\\
        &\quad-\frac{2}{\widetilde{R}^{\gamma+2}}\left|\widetilde{R}\cdot\nabla\widetilde{\rm Ric}-\nabla\widetilde{R}\cdot\widetilde{\rm Ric}\right|^{2}-\frac{(2-\gamma)(\gamma-1)}{\widetilde{R}^{\gamma+2}}|\widetilde{\rm Ric}|^{2}|\nabla\widetilde{R}|^{2}\notag\\
        &\quad-\frac{2\gamma}{\widetilde{R}^{\gamma+1}}|\widetilde{\rm Ric}|^{4}+\frac{1}{\widetilde{R}^{\gamma}}(I+J)-\frac{\gamma}{\widetilde{R}^{\gamma+1}}|\widetilde{\rm Ric}|^{2}H\notag,
    \end{align}
    and
    \begin{align}
        \partial_{t}\widetilde{R}^{2-\gamma}&=(2-\gamma)\widetilde{R}^{1-\gamma}\partial_{t}\widetilde{R}=(2-\gamma)\widetilde{R}^{1-\gamma}\left(\Delta\widetilde{R}+2|\widetilde{\rm Ric}|^{2}+H\right)\\
        &=(2-\gamma)\widetilde{R}^{1-\gamma}\Delta\widetilde{R}+2(2-\gamma)\widetilde{R}^{1-\gamma}|\widetilde{\rm Ric}|^{2}+(2-\gamma)\widetilde{R}^{1-\gamma}H\notag\\
        &=\Delta\widetilde{R}^{2-\gamma}-(2-\gamma)(1-\gamma)\widetilde{R}^{-\gamma}|\nabla\widetilde{R}|^{2}+2(2-\gamma)\widetilde{R}^{1-\gamma}|\widetilde{\rm Ric}|^{2}\notag\\
        &\quad+(2-\gamma)\widetilde{R}^{1-\gamma}H\notag,
    \end{align}
    where we use 
    $$\Delta\widetilde{R}^{2-\gamma}=\nabla_{i}\left((2-\gamma)\widetilde{R}^{1-\gamma}\cdot\nabla_{i}\widetilde{R}\right)=(2-\gamma)\widetilde{R}^{1-\gamma}\Delta\widetilde{R}+(2-\gamma)(1-\gamma)\widetilde{R}^{-\gamma}|\nabla\widetilde{R}|^{2}.$$
    Then we can calculate the evolution of $f$:
    \begin{align}
        \partial_{t}f&=\Delta f+\frac{2(\gamma-1)}{\widetilde{R}}\left\langle\nabla f,\nabla\widetilde{R}\right\rangle-\frac{2}{\widetilde{R}^{\gamma+2}}\left|\widetilde{R}\cdot\nabla\widetilde{\rm Ric}-\nabla\widetilde{R}\cdot\widetilde{\rm Ric}\right|^{2}\\
        &\quad-\frac{(2-\gamma)(\gamma-1)}{\widetilde{R}^{2}}|\nabla\widetilde{R}|^{2}f+\frac{4}{\widetilde{R}^{\gamma}}R_{pijl}\widetilde{R}^{pl}\widetilde{R}^{ij}-\frac{2\gamma}{\widetilde{R}^{\gamma+1}}|\widetilde{\rm Ric}|^{4}\notag\\
        &\quad+\frac{1}{\widetilde{R}^{\gamma}}(I+J)-\frac{\gamma}{\widetilde{R}^{\gamma+1}}|\widetilde{\rm Ric}|^{2}H-\frac{2}{7}\frac{(2-\gamma)}{\widetilde{R}^{\gamma-1}}|\widetilde{\rm Ric}|^{2}-\frac{(2-\gamma)}{7}\frac{H}{\widetilde{R}^{\gamma-1}}\notag\\
        &=\Delta f+\frac{2(\gamma-1)}{\widetilde{R}}\left\langle\nabla f,\nabla\widetilde{R}\right\rangle-\frac{2}{\widetilde{R}^{\gamma+2}}\left|\widetilde{R}\cdot\nabla\widetilde{\rm Ric}-\nabla\widetilde{R}\cdot\widetilde{\rm Ric}\right|^{2}\notag\\
        &\quad-\frac{(2-\gamma)(\gamma-1)}{\widetilde{R}^{2}}|\nabla\widetilde{R}|^{2}f+\frac{2}{\widetilde{R}^{\gamma+1}}\bigg{[}(2-\gamma)|\widetilde{\rm Ric}|^{2}\left(|\widetilde{\rm Ric}|^{2}-\frac{1}{7}\widetilde{R}^{2}\right)\notag\\
        &\quad-2\left(|\widetilde{\rm Ric}|^{4}-\widetilde{R}\cdot R_{pijl}\widetilde{R}^{pl}\widetilde{R}^{ij}\right)\bigg{]}+\frac{1}{\widetilde{R}^{\gamma}}(I+J)-\frac{\gamma}{\widetilde{R}^{\gamma+1}}|\widetilde{\rm Ric}|^{2}H\notag\\
        &\quad-\frac{(2-\gamma)}{7}\frac{H}{\widetilde{R}^{\gamma-1}}\notag.
    \end{align}
    Note that
    \begin{align}
        (2-\gamma)|\widetilde{\rm Ric}|^{2}\left(|\widetilde{\rm Ric}|^{2}-\frac{1}{7}\widetilde{R}^{2}\right)&=(2-\gamma)\left(|\widetilde{\rm Ric}|^{2}-\frac{1}{7}\widetilde{R}^{2}\right)|E|^{2}+\frac{2-\gamma}{7}\widetilde{R}^{2}|E|^{2}\notag\\
        &=(2-\gamma)|E|^{4}+\frac{2-\gamma}{7}\widetilde{R}^{2}|E|^{2}\notag,
    \end{align}
    \begin{align}
        -2|\widetilde{\rm Ric}|^{4}&=-2\Big(|\widetilde{\rm Ric}|^{4}-\frac{2}{7}\widetilde{R}^{2}|\widetilde{\rm Ric}|^{2}+\frac{1}{49}\widetilde{R}^{4}\Big)-\frac{4}{7}\widetilde{R}^{2}|\widetilde{\rm Ric}|^{2}+\frac{2}{49}\widetilde{R}^{4}\notag\\
        &=-2|E|^{4}-\frac{4}{7}\widetilde{R}^{2}\Big(|\widetilde{\rm Ric}|^{2}-\frac{1}{7}\widetilde{R}^{2}\Big)-\frac{4}{49}\widetilde{R}^{4}+\frac{2}{49}\widetilde{R}^{4}\notag\\
        &=-2|E|^{4}-\frac{4}{7}\widetilde{R}^{2}|E|^{2}-\frac{2}{49}\widetilde{R}^{4}\notag,\\
        2\widetilde{R}\cdot R_{pijl}\widetilde{R}^{pl}\widetilde{R}^{ij}&=2\widetilde{R}\cdot W_{pijl}\widetilde{R}^{pl}\widetilde{R}^{ij}+\frac{2}{5}\widetilde{R}\widetilde{R}^{pl}\widetilde{R}^{ij}\left(g_{pl}R_{ij}+g_{ij}R_{pl}\right.\notag\\
        &\quad\left.-g_{pj}R_{il}-g_{il}R_{pj}\right)-\frac{1}{15}R\widetilde{R}\widetilde{R}^{pl}\widetilde{R}^{ij}(g_{pl}g_{ij}-g_{pj}g_{il})\notag\\
        &=2\widetilde{R}\cdot W_{pijl}\widetilde{R}^{pl}\widetilde{R}^{ij}+\frac{4}{5}\widetilde{R}^{2}\widetilde{R}^{ij}\left(\widetilde{R}_{ij}-\frac{c}{7}g_{ij}\right)\notag\\
        &\quad-\frac{4}{5}\widetilde{R}\widetilde{R}_{j}^{\ l}\widetilde{R}^{ij}\left(\widetilde{R}_{il}-\frac{c}{7}g_{il}\right)-\frac{1}{15}\widetilde{R}^{3}\left(\widetilde{R}-c\right)\notag\\
        &\quad+\frac{1}{15}\left(\widetilde{R}-c\right)\widetilde{R}|\widetilde{\rm Ric}|^{2}\notag\\
        &=2\widetilde{R}\cdot W_{pijl}E^{pl}E^{ij}+\frac{13}{15}\widetilde{R}^{2}|\widetilde{\rm Ric}|^{2}-\frac{4}{5}\widetilde{R}\widetilde{\rm Ric}{}^{3}\notag\\
        &\quad+\frac{c}{21}\widetilde{R}|\widetilde{\rm Ric}|^{2}-\frac{c}{21}\widetilde{R}^{3}-\frac{1}{15}\widetilde{R}^{4}\notag,
    \end{align}
    where the last equality uses the traceless of Weyl tensor and $\widetilde{\rm Ric}{}^{3}=\widetilde{R}_{ij}\widetilde{R}^{j}_{\ l}\widetilde{R}^{li}$. Since
    \begin{align}
        \frac{13}{15}\widetilde{R}^{2}|\widetilde{\rm Ric}|^{2}&=\frac{13}{15}\widetilde{R}^{2}\left(|\widetilde{\rm Ric}|^{2}-\frac{1}{7}\widetilde{R}^{2}\right)+\frac{13}{105}\widetilde{R}^{4}=\frac{13}{15}\widetilde{R}^{2}|E|^{2}+\frac{13}{105}\widetilde{R}^{4}\notag,\\
        -\frac{4}{5}\widetilde{R}\widetilde{\rm Ric}{}^{3}&=-\frac{4}{5}\widetilde{R}\left(\widetilde{\rm Ric}{}^{3}-\frac{3}{7}\widetilde{R}|\widetilde{\rm Ric}|^{2}+\frac{2}{49}\widetilde{R}{}^{3}\right)-\frac{12}{35}\widetilde{R}^{2}|\widetilde{\rm Ric}|^{2}+\frac{8}{245}\widetilde{R}^{4}\notag\\
        &=-\frac{4}{5}\widetilde{R}E^{3}-\frac{12}{35}\widetilde{R}^{2}\left(|\widetilde{\rm Ric}|^{2}-\frac{1}{7}\widetilde{R}^{2}\right)-\frac{12}{245}\widetilde{R}^{4}+\frac{8}{245}\widetilde{R}^{4}\notag\\
        &=-\frac{4}{5}\widetilde{R}E^{3}-\frac{12}{35}\widetilde{R}^{2}|E|^{2}-\frac{4}{245}\widetilde{R}^{4}\notag,\\
        \frac{c}{21}\widetilde{R}|\widetilde{\rm Ric}|^{2}&=\frac{c}{21}\widetilde{R}\left(|\widetilde{\rm Ric}|^{2}-\frac{1}{7}\widetilde{R}^{2}\right)+\frac{c}{147}\widetilde{R}^{3}=\frac{c}{21}\widetilde{R}|E|^{2}+\frac{c}{147}\widetilde{R}^{3}\notag.
    \end{align}
    In summary, we get
    \begin{align}
        \partial_{t}f&=\Delta f+\frac{2(\gamma-1)}{\widetilde{R}}\left\langle\nabla f,\nabla\widetilde{R}\right\rangle-\frac{2}{\widetilde{R}^{\gamma+2}}\left|\widetilde{R}\cdot\nabla\widetilde{\rm Ric}-\nabla\widetilde{R}\cdot\widetilde{\rm Ric}\right|^{2}\\
        &\quad-\frac{(2-\gamma)(\gamma-1)}{\widetilde{R}^{2}}|\nabla\widetilde{R}|^{2}f+\frac{2}{\widetilde{R}^{\gamma+1}}\bigg[-\gamma|E|^{4}+2\widetilde{R}\cdot W_{pijl}E^{pl}E^{ij}\notag\\
        &\quad-\frac{4}{5}\widetilde{R}E^{3}+\left(\frac{5}{21}-\frac{\gamma}{7}\right)\widetilde{R}^{2}|E|^{2}+\frac{c}{21}\widetilde{R}|E|^{2}-\frac{2c}{49}\widetilde{R}^{3}\bigg]\notag\\
        &\quad+\frac{1}{\widetilde{R}^{\gamma}}(I+J)-\frac{\gamma}{\widetilde{R}^{\gamma+1}}|\widetilde{\rm Ric}|^{2}H
        -\frac{(2-\gamma)}{7}\frac{H}{\widetilde{R}^{\gamma-1}}\notag.
    \end{align}
    Thus we get the desired result.
\end{proof}

If we let $\gamma=2$, which means
$$f=\frac{|E|^{2}}{\widetilde{R}^{2}}=\frac{|\widetilde{\rm Ric}|^{2}}{\widetilde{R}^{2}}-\frac{1}{7},$$
then we have the following

\begin{lemma}\label{lemma4.2}
    Under the Laplacian $G_{2}$ flow $\eqref{The closed Laplacian flow}$. If there exists a positive constant $c$ independent on time $t$ such that $\widetilde{R}=R+c>0$ {and any small positive constant $\delta$ such that $f\geq \delta$}, then we have
   \begin{align}
       \partial_{t}f&\leq\Delta f+\frac{2}{\widetilde{R}}\left\langle\nabla f,\nabla\widetilde{R}\right\rangle+4\widetilde{R}f\left(-\frac{1}{2}f+Cf^{\frac{1}{2}}+C+C\frac{|W|_{C^{1}(M,g(t))}^{2}}{\widetilde{R}^{2}}\right),
   \end{align}
    where $C$  are any positive constants only dependent on $c$.
\end{lemma}

\begin{proof}
    Letting $\gamma=2$ in Lemma \ref{lemma4.1}, the evolution equation of $f$ is
    \begin{align}
         \partial_{t}f&=\Delta f+\frac{2}{\widetilde{R}}\left\langle\nabla f,\nabla\widetilde{R}\right\rangle-\frac{2}{\widetilde{R}^{4}}\left|\widetilde{R}\cdot\nabla\widetilde{\rm Ric}-\nabla\widetilde{R}\cdot\widetilde{\rm Ric}\right|^{2}+\frac{2}{\widetilde{R}^{3}}\bigg[-2|E|^{4}\\
        &\quad+2\widetilde{R}\cdot W_{pijl}E^{pl}E^{ij}-\frac{4}{5}\widetilde{R}E^{3}-\frac{1}{21}\widetilde{R}^{2}|E|^{2}+\frac{c}{21}\widetilde{R}|E|^{2}-\frac{2c}{49}\widetilde{R}^{3}\bigg]\notag\\
        &\quad+\frac{1}{\widetilde{R}^{2}}(I+J)-\frac{\gamma}{\widetilde{R}^{3}}|\widetilde{\rm Ric}|^{2}H
        \notag\\
        &=\Delta f+\frac{2}{\widetilde{R}}\left\langle\nabla f,\nabla\widetilde{R}\right\rangle-\frac{2}{\widetilde{R}^{4}}\left|\widetilde{R}\cdot\nabla\widetilde{\rm Ric}-\nabla\widetilde{R}\cdot\widetilde{\rm Ric}\right|^{2}\notag\\
        &\quad+4\widetilde{R}\left[-f^{2}-\frac{2}{5}\frac{E^{3}}{\widetilde{R}^{3}}-\frac{1}{42}f+\frac{1}{42}\frac{c}{\widetilde{R}}f-\frac{1}{49}\frac{c}{\widetilde{R}}+\frac{1}{\widetilde{R}^{3}}W_{pijl}E^{pl}E^{ij}\right]\notag\\
        &\quad+\frac{1}{\widetilde{R}^{2}}(I+J)-\frac{2}{\widetilde{R}^{3}}|\widetilde{\rm Ric}|^{2}H
        \notag,
    \end{align}
    To estimate the above equation, we claim that there exist some constants $C$, such that
    $$
    \left|\frac{2}{5}E^{3}\right|\leq C|E|^{3},\ \ \  |W(E,E)|\leq C|W||E|^{2},
    $$
    Then we have
    \begin{align}
        \partial_{t}f&\leq\Delta f+\frac{2}{\widetilde{R}}\left\langle\nabla f,\nabla\widetilde{R}\right\rangle-\frac{2}{\widetilde{R}^{4}}\left|\widetilde{R}\cdot\nabla\widetilde{\rm Ric}-\nabla\widetilde{R}\cdot\widetilde{\rm Ric}\right|^{2}\notag\\
        &\quad+4\widetilde{R}\left[-f^{2}+Cf^{\frac{3}{2}}-\frac{1}{42}f+\frac{1}{42}\frac{c}{\widetilde{R}}f-\frac{1}{49}\frac{c}{\widetilde{R}}+C\frac{|W|}{\widetilde{R}}f\right]\notag\\
        &\quad+\frac{1}{\widetilde{R}^{2}}(I+J)-\frac{2}{\widetilde{R}^{3}}|\widetilde{\rm Ric}|^{2}H
        \notag\\
        &\leq \Delta f+\frac{2}{\widetilde{R}}\left\langle\nabla f,\nabla\widetilde{R}\right\rangle-\frac{2}{\widetilde{R}^{4}}\left|\widetilde{R}\cdot\nabla\widetilde{\rm Ric}-\nabla\widetilde{R}\cdot\widetilde{\rm Ric}\right|^{2}\notag\\
        &\quad +4\widetilde{R}\left(-f^{2}+Cf^{\frac{3}{2}}+Cf+C+C\frac{|W|^{2}}{\widetilde{R}^{2}}f\right)+\frac{1}{\widetilde{R}^{2}}(I+J)-\frac{2}{\widetilde{R}^{3}}|\widetilde{\rm Ric}|^{2}H\notag.
    \end{align}
    
    Next we estimate the last two terms. The term involving $I$ can be estimated as follows.
    \begin{align}\label{4.91}
        \frac{1}{\widetilde{R}^{2}}I&=\frac{1}{\widetilde{R}^{2}}\left[-\frac{4}{3}\widetilde{R}|\widetilde{\rm Ric}|^{2}+\frac{16c}{21}|\widetilde{\rm Ric}|^{2}+8R_{pijl}\widehat{\mathbf{T}}^{pl}\widetilde{R}^{ij}+\frac{4c}{21}\widetilde{R}^{2}-\frac{16}{147}c^{2}\widetilde{R}\right]\\
        &=\frac{1}{\widetilde{R}^{2}}\left[-\frac{4}{3}\widetilde{R}\left(|\widetilde{\rm Ric}|^{2}-\frac{1}{7}\widetilde{R}^{2}\right)-\frac{4}{21}\widetilde{R}^{3}+\frac{16c}{21}\left(|\widetilde{\rm Ric}|^{2}-\frac{1}{7}\widetilde{R}^{2}\right)+\frac{16c}{147}\widetilde{R}^{2}\right.\notag\\
        &\quad\left.+8R_{pijl}\widehat{\mathbf{T}}^{pl}\widetilde{R}^{ij}+\frac{4c}{21}\widetilde{R}^{2}-\frac{16}{147}c^{2}\widetilde{R}\right]\notag\\
        &=\frac{1}{\widetilde{R}^{2}}\left[-\frac{4}{3}\widetilde{R}\ |E|^{2}+\frac{16c}{21}|E|^{2}-\frac{4}{21}\widetilde{R}^{3}+\frac{44c}{147}\widetilde{R}^{2} -\frac{16}{147}c^{2}\widetilde{R}+8R_{pijl}\widehat{\mathbf{T}}^{pl}\widetilde{R}^{ij}\right]\notag\\
        &=4\widetilde{R}\left(-\frac{1}{3}f+\frac{4}{21}\frac{c}{\widetilde{R}}f-\frac{1}{21}+\frac{11}{147}\frac{c}{\widetilde{R}}-\frac{4}{147}\frac{c^{2}}{\widetilde{R}^{2}}\right)+\frac{8}{\widetilde{R}^{2}}R_{pijl}\widehat{\mathbf{T}}^{pl}\widetilde{R}^{ij}\notag\\
        &\leq 4\widetilde{R}\left(Cf+Cf^{\frac{1}{2}}+C+C\frac{|W|^{2}}{\widetilde{R}^{2}}f+C\frac{|W|^{2}}{\widetilde{R}^{2}}\right)\notag.
    \end{align}
    where we use $|\mathbf{T}|^{2}=-R=-(\widetilde{R}-c)=-\widetilde{R}+c\leq c$ and
    \begin{align}\label{4.9}
        {\rm Rm}&=W+{\rm Ric}\ast g+R\ast g\ast g=W+\widetilde{\rm Ric}\ast g+\widetilde{R}\ast g\ast g+g\ast g\\
        &=W+E\ast g+\widetilde{R}\ast g\ast g+g\ast g\notag,\\\label{4.10}
        |\widetilde{\rm Ric}|&=\left|\left(\widetilde{\rm Ric}-\frac{1}{7}\widetilde{R}g\right)+\frac{1}{7}\widetilde{R}g\right|\leq C|E|+C\widetilde{R},\\
        \frac{1}{\widetilde{R}^{2}}|W||\mathbf{T}|^{2}|\widetilde{\rm Ric}|&\leq\frac{c}{\widetilde{R}^{2}}|W||\widetilde{\rm Ric}|\notag\leq \frac{C}{\widetilde{R}}|W|f^{\frac{1}{2}}+\frac{C}{\widetilde{R}}|W|\notag,\\
        \frac{1}{\widetilde{R}^{2}}|E||\mathbf{T}|^{2}|\widetilde{\rm Ric}|&\leq Cf+Cf^{\frac{1}{2}}\notag,\\
        \frac{1}{\widetilde{R}^{2}}\cdot\widetilde{R}|\mathbf{T}|^{2}|\widetilde{\rm Ric}|&\leq Cf^{\frac{1}{2}}+C\notag,\\
        \frac{1}{\widetilde{R}^{2}}|\mathbf{T}|^{2}|\widetilde{\rm Ric}|&\leq \frac{C}{\widetilde{R}}f^{\frac{1}{2}}+\frac{C}{\widetilde{R}}\notag.
    \end{align}
For the term contains $H$, we obtain
    \begin{align}\label{4.12}
        -\frac{2}{\widetilde{R}^{3}}|\widetilde{\rm Ric}|^{2}H&=-\frac{2}{\widetilde{R}^{3}}|\widetilde{\rm Ric}|^{2}\bigg(-\frac{2}{3}\widetilde{R}^{2}+\frac{16c}{21}\widetilde{R}-\frac{8}{21}c^{2}+4R_{ijmn}\mathbf{T}^{in}\mathbf{T}^{mj}\\
        &\quad-4\nabla^{j}\mathbf{T}_{im}\cdot\nabla^{i}\mathbf{T}^{m}_{\ \ j}\bigg)\notag\\
        &=-\frac{2}{\widetilde{R}^{3}}\left(|E|^{2}+\frac{1}{7}\widetilde{R}^{2}\right)\left(-\frac{2}{3}\widetilde{R}^{2}+\frac{16c}{21}\widetilde{R}-\frac{8}{21}c^{2}\right)\notag\\
        &\quad-\frac{2}{\widetilde{R}^{3}}|\widetilde{\rm Ric}|^{2}\left(4R_{ijmn}\mathbf{T}^{in}\mathbf{T}^{mj}-4\nabla^{j}\mathbf{T}_{im}\cdot\nabla^{i}\mathbf{T}^{m}_{\ \ j}\right)\notag\\
        &=4\widetilde{R}\left(\frac{1}{3}f-\frac{8}{21}\frac{c}{\widetilde{R}}f+\frac{4}{21}\frac{c^{2}}{\widetilde{R}^{2}}f+\frac{8}{147}\frac{c}{\widetilde{R}}-\frac{4}{147}\frac{c^{2}}{\widetilde{R}^{2}}+\frac{1}{21}\right)\notag\\
        &\quad-\frac{2}{\widetilde{R}^{3}}|\widetilde{\rm Ric}|^{2}\left(4R_{ijmn}\mathbf{T}^{in}\mathbf{T}^{mj}-4\nabla^{j}\mathbf{T}_{im}\cdot\nabla^{i}\mathbf{T}^{m}_{\ \ j}\right)\notag.
    \end{align}
    Using {$\eqref{4.9}$-$\eqref{4.10}$} and Cauchy-Schwarz inequality yields
    \begin{align}
        \frac{|\widetilde{\rm Ric}|^{2}}{\widetilde{R}^{3}}|W||\mathbf{T}|^{2}&\leq \frac{C}{\widetilde{R}}|W||\mathbf{T}|^{2}f+\frac{C}{\widetilde{R}}|W||\mathbf{T}|^{2}\leq \frac{C}{\widetilde{R}}|W|f+\frac{C}{\widetilde{R}}|W|\notag,\\
        \frac{|\widetilde{\rm Ric}|^{2}}{\widetilde{R}^{3}}|E||\mathbf{T}|^{2}&\leq Cf^{\frac{3}{2}}+Cf^{\frac{1}{2}},\notag\\
        \frac{|\widetilde{\rm Ric}|^{2}}{\widetilde{R}^{3}}\cdot\widetilde{R}|\mathbf{T}|^{2}&\leq Cf+C,\notag\\
        \frac{|\widetilde{\rm Ric}|^{2}}{\widetilde{R}^{3}}|\mathbf{T}|^{2}&\leq \frac{C}{\widetilde{R}}f+\frac{C}{\widetilde{R}}\notag.
    \end{align}    
    From Page 15-16 in \cite{local curvature of G2}, we have
    \begin{align}
        -\nabla^{j}\mathbf{T}_{im}\cdot\nabla^{i}\mathbf{T}^{m}_{\ \ j}&=|\nabla \mathbf{T}|^{2}-\frac{1}{4}|{\rm Rm}|^{2}-\frac{1}{8}R_{ijab}R^{ijmn}\psi_{mn}^{\ \ \ ab}+{\rm Rm}\ast \mathbf{T}\ast \mathbf{T}\notag\\
        &\quad-\frac{1}{2}|\mathbf{T}|^{4}+\frac{1}{2}|\widehat{\mathbf{T}}|^{2}+{\rm Rm}\ast \mathbf{T}\ast \mathbf{T}\ast\psi+\mathbf{T}\ast \mathbf{T}\ast \mathbf{T}\ast \mathbf{T}\ast\psi\notag.
    \end{align}
    Next we only consider the terms may contain $|{\rm Ric}|^{2}$, which are $R_{ijab}R^{ijmn}\psi_{mn}^{\ \ \ ab}$ and $|{\rm Rm}|^{2}$. Then we have the following:
    \begin{align}
        &\quad\frac{1}{25}\left(g_{il}R_{jk}+g_{jk}R_{il}-g_{ik}R_{jl}-g_{jl}R_{ik}\right)\left(g^{il}R^{jk}+g^{jk}R^{il}-g^{ik}R^{jl}-g_{jl}R^{ik}\right)\notag\\
        &=\frac{4}{5}|{\rm Ric}|^{2}+\frac{4}{25}R^{2}\notag,
    \end{align}
    Hence, Using Young's inequality, we get
    \begin{align}
    |{\rm Rm}|^{2}&=|W|^{2}+\frac{4}{5}|{\rm Ric}|^{2}+{\rm Ric}\ast R\ast g\ast g\ast g+R\ast R\ast g\ast g\ast g\ast g\notag\\
    &\leq |W|^{2}+\frac{4}{5}|{\rm Ric}|^{2}+C|{\rm Ric}||R|+C|R|^{2}\notag\\
    &\leq |W|^{2}+\frac{9}{10}|\widetilde{\rm Ric}|^{2}+C\widetilde{R}^{2}+C\notag,
    \end{align}
and since $\psi_{ijkl}$ is skew-symmetric for any two indices, then
\begin{align}
        &\quad\frac{1}{25}\left(g_{ib}R_{ja}+g_{ja}R_{ib}-g_{ia}R_{jb}-g_{jb}R_{ia}\right)\notag\\
        &\quad\cdot\left(g^{in}R^{jm}+g^{jm}R^{in}-g^{im}R^{jn}-g_{jn}R^{im}\right)\psi_{mn}^{\ \ \ ab}\notag\\
        &=\frac{4}{25}\left(R_{ja}R^{jm}\delta_{b}^{n}+R_{\ \ a}^{m}R_{b}^{\ n}+R_{a}^{\ m}R_{\ b}^{n}+R_{ib}R^{in}\delta_{a}^{m}\right)\psi_{mn}^{\ \ \ ab}\notag\\
        &=0,\notag
        \end{align}
    which means
    \begin{align}
        \left|R_{ijab}R^{ijmn}\psi_{mn}^{\ \ \ ab}\right|&\leq C|W|^{2}+C|W||{\rm Ric}|+C|R||W|+C|R||{\rm Ric}|+CR^{2}\notag\\
        &\leq C|W|^{2}+\frac{1}{10}|\widetilde{\rm Ric}|^{2}+C\widetilde{R}^{2}+C\notag,
    \end{align}
    Together with the following:
    \begin{align}
        \frac{2}{\widetilde{R}^{3}}|\widetilde{\rm Ric}|^{2}|W|^{2}&\leq\frac{2}{\widetilde{R}^{3}}|E|^{2}|W|^{2}+\frac{C}{\widetilde{R}}|W|^{2}\leq\frac{2}{\widetilde{R}}|W|^{2}f+\frac{C}{\widetilde{R}}|W|^{2}\notag,\\
        \frac{2}{\widetilde{R}^{3}}|\widetilde{\rm Ric}|^{2}\cdot|\widetilde{\rm Ric}|^{2}&\leq 2\widetilde{R}f^{2}+C\widetilde{R}f+C\widetilde{R}\notag,\\
          \frac{2}{\widetilde{R}^{3}}|\widetilde{\rm Ric}|^{2}\left(C\widetilde{R}^{2}+C\right)&\leq C\widetilde{R}f+\frac{C}{\widetilde{R}}f+\frac{C}{\widetilde{R}}+C\widetilde{R}\notag.
    \end{align}
    we have
    \begin{align}
        -\frac{2}{\widetilde{R}^{3}}|\widetilde{\rm Ric}|^{2}H
        &\leq 4\widetilde{R}\left(\frac{1}{2}f^{2}+Cf^{\frac{3}{2}}+Cf+Cf^{\frac{1}{2}}+C+C\frac{|W|^{2}}{\widetilde{R}^{2}}f+C\frac{|W|^{2}}{\widetilde{R}^{2}}\right).
    \end{align}

    For the term containing $J$, we need to calculate the following four terms:
    $$\frac{1}{\widetilde{R}^{2}}J=\frac{1}{\widetilde{R}^{2}}\left(\frac{2}{3}\widetilde{R}\Delta|\mathbf{T}|^{2}+4\widetilde{R}^{ij}\Delta\widehat{\mathbf{T}}_{ij}-\frac{2}{3}\widetilde{R}^{ij}\nabla_{i}\nabla_{j}|\mathbf{T}|^{2}-8\widetilde{R}^{ij}\nabla_{i}\nabla^{p}\widehat{\mathbf{T}}_{pj}\right)$$
    From $\eqref{2.7}$ and Lemma \ref{lemma2.4}, we get
    $$\nabla \varphi=\mathbf{T}\ast\psi,\ \ \ \ \nabla \mathbf{T}={\rm Rm}\ast\varphi+\mathbf{T}\ast \mathbf{T}\ast\varphi,$$
    which imply
    \begin{align}
       |\nabla \mathbf{T}|&\leq C|{\rm Rm}|+C|\mathbf{T}|^{2}\leq C|{\rm Rm}|,\\
        |\nabla \mathbf{T}|^{2}&\leq C|W|^{2}+C|{\rm Ric}|^{2}+CR^{2}\leq C|W|^{2}+C|\widetilde{\rm Ric}|^{2}+C\widetilde{R}^{2}+C,\notag\\
        \nabla^{2}\mathbf{T}&=\nabla{\rm Rm}\ast\varphi+{\rm Rm}\ast\nabla\varphi+\nabla \mathbf{T}\ast \mathbf{T}\ast\varphi+\mathbf{T}\ast \mathbf{T}\ast\nabla\varphi\notag\\
        &=\nabla{\rm Rm}\ast\varphi+{\rm Rm}\ast \mathbf{T}\ast\psi+{\rm Rm}\ast \mathbf{T}\ast\varphi\ast\varphi+\mathbf{T}\ast \mathbf{T}\ast \mathbf{T}\ast\varphi\ast\varphi\notag\\
        &\quad+\mathbf{T}\ast \mathbf{T}\ast \mathbf{T}\ast\psi\notag\\
        &\leq C|\nabla{\rm Rm}|+C|{\rm Rm}||\mathbf{T}|+C|\mathbf{T}|^{3}.\notag
    \end{align}
Note that
    \begin{align}
        \nabla{\rm Rm}&=\nabla W+\nabla{\rm Ric}\ast g+\nabla R\ast g\ast g\notag\\
        &=\nabla W+\nabla\widetilde{\rm Ric}\ast g+\nabla \widetilde{R}\ast g\ast g\notag\\
        &=\nabla W+\frac{1}{\widetilde{R}}\ast\left(\widetilde{R}\nabla\widetilde{\rm Ric}-\nabla\widetilde{R}\widetilde{\rm Ric}\right)\ast g+\frac{1}{\widetilde{R}}\ast\nabla\widetilde{R}\ast \widetilde{\rm Ric}\ast g+\nabla \widetilde{R}\ast g\ast g\notag,
    \end{align}
    we obtain
    \begin{align}
        \frac{1}{\widetilde{R}^{2}}\frac{2}{3}\widetilde{R}\Delta|\mathbf{T}|^{2}&\leq\frac{C}{\widetilde{R}}|\nabla \mathbf{T}|^{2}+\frac{C}{\widetilde{R}}\bigg(|\nabla W||\mathbf{T}|+\frac{1}{\widetilde{R}}|\widetilde{\rm Ric}||\nabla\widetilde{R}||\mathbf{T}|+|\nabla\widetilde{R}||\mathbf{T}|\\
        &\quad+|{\rm Rm}||\mathbf{T}|^{2}+|\mathbf{T}|^{4}\bigg)+\frac{C}{\widetilde{R}^{2}}\left|\widetilde{R}\nabla\widetilde{\rm Ric}-\nabla\widetilde{R}\widetilde{\rm Ric}\right|\cdot|\mathbf{T}|\notag\\
        &\leq\frac{1}{5}\frac{1}{\widetilde{R}^{4}}\left|\widetilde{R}\cdot\nabla\widetilde{\rm Ric}-\nabla\widetilde{R}\cdot\widetilde{\rm Ric}\right|^{2}+\frac{C}{\widetilde{R}^{2}}|{\rm Rm}|^{2}+\frac{C}{\widetilde{R}^{2}}|\nabla\widetilde{R}|^{2}\notag\\
        &\quad+\frac{C}{\widetilde{R}^{2}}|\widetilde{\rm Ric}|^{2}|\mathbf{T}|^{2}+\frac{C}{\widetilde{R}^{2}}|\nabla W|^{2}+\frac{C}{\widetilde{R}}|\nabla \mathbf{T}|^{2}+C|\mathbf{T}|^{4}+C|\mathbf{T}|^{2}.\notag
    \end{align}
    Recalling the definition of curvature decomposition, which yields
    \begin{align}
        |{\rm Rm}|^{2}&\leq C|W|^{2}+C|{\rm Ric}|^{2}+CR^{2}\leq C|W|^{2}+C|\widetilde{\rm Ric}|^{2}+C\widetilde{R}^{2}+C\notag\\
        |\nabla\widetilde{R}|&=\left|\nabla|\mathbf{T}|^{2}\right|\leq |\mathbf{T}||\nabla \mathbf{T}|\notag.
    \end{align}
    By calculating the following estimates, we have
    \begin{align}
        \frac{1}{\widetilde{R}^{2}}|{\rm Rm}|^{2}&\leq C\frac{|W|^{2}}{\widetilde{R}^{2}}+Cf+\frac{C}{\widetilde{R}^{2}}+C\notag,\\
        \frac{1}{\widetilde{R}^{2}}|\nabla\widetilde{R}|^{2}&\leq\frac{c}{\widetilde{R}^{2}}|\nabla \mathbf{T}|^{2}\leq\frac{C}{\widetilde{R}^{2}}|\rm Rm|^{2},\notag\\
        \frac{1}{\widetilde{R}^{2}}|\widetilde{\rm Ric}|^{2}|\mathbf{T}|^{2}&\leq Cf+C,\notag\\
         \frac{1}{\widetilde{R}^{2}}\left(|\mathbf{T}|^{4}+|\mathbf{T}|^{2}\right)&\leq \frac{C}{\widetilde{R}^{2}}.\notag
    \end{align}
  Thus the first term can be estimated as follows:
  \begin{align}
      \frac{1}{\widetilde{R}^{2}}\frac{2}{3}\widetilde{R}\Delta|\mathbf{T}|^{2}&\leq\frac{1}{5}\frac{1}{\widetilde{R}^{4}}\left|\widetilde{R}\cdot\nabla\widetilde{\rm Ric}-\nabla\widetilde{R}\cdot\widetilde{\rm Ric}\right|^{2}\\
      &\quad+4R\left(\frac{C\big{(}|W|^{2}+|\nabla W|^{2}\big{)}}{\widetilde{R}^{2}}+Cf+C\right)\notag.
  \end{align}
  Similar with (4.16), we have
  \begin{align}
  &\quad\frac{1}{\widetilde{R}^{2}}\left(4\widetilde{R}^{ij}\Delta\widehat{\mathbf{T}}_{ij}-\frac{2}{3}\widetilde{R}^{ij}\nabla_{i}\nabla_{j}|\mathbf{T}|^{2}-8\widetilde{R}^{ij}\nabla_{i}\nabla^{p}\widehat{\mathbf{T}}_{pj}\right)\\
  &\leq\frac{C}{\widetilde{R}^{2}}|\widetilde{\rm Ric}|\left(|\nabla^{2}\mathbf{T}||\mathbf{T}|+|\nabla \mathbf{T}|^{2}\right)\notag\\
  &\leq\frac{C}{\widetilde{R}}\left(f^{\frac{1}{2}}+C\right)\left(|\nabla^{2}\mathbf{T}||\mathbf{T}|+|\nabla \mathbf{T}|^{2}\right)\notag\\
  &\leq\frac{1}{5}\frac{1}{\widetilde{R}^{4}}\left|\widetilde{R}\cdot\nabla\widetilde{\rm Ric}-\nabla\widetilde{R}\cdot\widetilde{\rm Ric}\right|^{2}+4R\left(C(f^{\frac{1}{2}}+1)\frac{|W|^{2}+|\nabla W|^{2}}{\widetilde{R}^{2}}\right.\notag\\
  &\quad+Cf^{\frac{3}{2}}+Cf^{\frac{1}{2}}+Cf+C\bigg)\notag.
  \end{align}
  Then we can conclude that
  \begin{align}\label{4.18}
      \frac{1}{\widetilde{R}^{2}}J
  &\leq \frac{2}{5}\frac{1}{\widetilde{R}^{4}}\left|\widetilde{R}\cdot\nabla\widetilde{\rm Ric}-\nabla\widetilde{R}\cdot\widetilde{\rm Ric}\right|^{2}+4\widetilde{R}\left(Cf^{\frac{1}{2}}\frac{|W|^{2}+|\nabla W|^{2}}{\widetilde{R}^{2}}\right.\\
  &\quad\left.+C\frac{|W|^{2}+|\nabla W|^{2}}{\widetilde{R}^{2}}+Cf^{\frac{3}{2}}+Cf+Cf^{\frac{1}{2}}+C\right).\notag
  \end{align}
  
  Without loss of generality, {by scaling,} we may assume that $f\geq 1$. Combining with $\eqref{4.91}$, $\eqref{4.12}$, and $\eqref{4.18}$, we have
  \begin{align}
       \partial_{t}f&\leq\Delta f+\frac{2}{\widetilde{R}}\left\langle\nabla f,\nabla\widetilde{R}\right\rangle+4\widetilde{R}\left(-\frac{1}{2}f^{2}+Cf^{\frac{3}{2}}+Cf+C\frac{|W|_{C^{1}(M)}^{2}}{\widetilde{R}^{2}}f\right)\notag\\
       &=\Delta f+\frac{2}{\widetilde{R}}\left\langle\nabla f,\nabla\widetilde{R}\right\rangle+4\widetilde{R}f\left(-\frac{1}{2}f+Cf^{\frac{1}{2}}+C+C\frac{|W|_{C^{1}(M)}^{2}}{\widetilde{R}^{2}}\right)\notag.
  \end{align}
  where $C$ is a positive constant depending on $c$.
\end{proof}

\begin{theorem}\label{theorem4.3}
    Let $(M,\varphi(t))_{t\in[0,T)}$ be the solution of the Laplacian $G_{2}$ flow $\eqref{The closed Laplacian flow}$ on a closed $7$-dimensional manifold $M$ with $T<+\infty$, where $g(t)$ is the Riemannian metric associated with $\varphi(t)$. If there exists a positive constant $c$ independent on time $t$ such that $R(g(t))+c>0$, then there exist some constants $c_{1}=c_{1}(c,\varphi(0))\geq0$ with $f(0)\leq c_{1}^{2}$ , {$C_{1}=C_{1}(c_{1})\geq0$} and $ C_{2}=C_{2}(c)\geq0$, such that for all $t\in [0,T)$, one has 
    \begin{align}\label{4.19}
        \frac{|E(g(t))|_{g(t)}}{R(g(t))+c}\leq C_{1}+C_{2}\max_{M\times[0,t]}\frac{|W|_{C^{1}(M,g(t))}}{R(g(t))+c},
    \end{align}
    where $\displaystyle{E(g(t))={\rm Ric}(g(t))-\frac{1}{7}R(g(t))g(t)}$ is the Einstein tensor of $g(t)$.
\end{theorem}
\begin{proof}
   {Fix a time $t\in [0,T)$, for any points $(x,t)\in M\times[0,t]$ such that $f(x,t)< \delta$($\delta$ is a small positive constant satisfies $\delta<c$). The pinching equality holds $\eqref{4.19}$ automatically. We now consider points satisfying $f(x,t)\geq \delta$.}
  If we assume
    $$-\frac{1}{2}f+Cf^{\frac{1}{2}}+C+C\frac{|W|_{C^{1}(M,g(t))}^{2}}{\widetilde{R}^{2}}\leq 0,$$
    then applying the maximum principle to Lemma \ref{lemma4.2}, we get $f(t)\leq c_{1}^{2}$, Otherwise, we have
    $$-\frac{1}{2}f+Cf^{\frac{1}{2}}+C+C\frac{|W|_{C^{1}(M,g(t))}^{2}}{\widetilde{R}^{2}}> 0,$$
    which means 
    $$f^{\frac{1}{2}}\leq 2C+1+C_{2}\max_{M\times[0,t]}\frac{|W|_{C^{1}(M,g(t))}}{R+c}.$$
    Letting $C_{1}=\max\left\{c_{1},{2C+1}\right\}$, we prove this theorem.
\end{proof}

\begin{theorem}\label{theorem4.4}
     Let $(M,\varphi(t))_{t\in[0,T)}$ be the solution of the Laplacian $G_{2}$ flow $\eqref{The closed Laplacian flow}$ on a closed $7$-dimensional manifold $M$ with $T<+\infty$, where $g(t)$ is the Riemannian metric associated with $\varphi(t)$. Either one has
     $$\liminf_{t\rightarrow T}\left(\min_{M} R(g(t))\right)=-\infty$$
      or (for some positive constant $c$)
     $$R(g(t))+c>0 \ \text{on} \ M\times[0,T) \ \ \ \text{but}\ \ \limsup_{t\rightarrow T}\left(\max_{M}\frac{|W(g(t))|_{C^{1}(M,g(t))}}{R(g(t))+c}\right)=+\infty.$$
\end{theorem}
\begin{proof}
    Let us assume that $R(g(t))$ has a uniformly lower bound, which means for some positive constant $c$, we have
    $$R(g(t))+c>0 \ \text{on} \ M\times[0,T).$$
    If we assume
    $$\limsup_{t\rightarrow T}\left(\max_{M}\frac{|W|_{C^{1}(M,g(t))}}{R(g(t))+c}\right)<+\infty,$$
    then Theorem \ref{theorem4.3} implies that $E$ is uniformly bounded. From the definition of curvature decomposition, we have that the Riemann curvature ${\rm Rm}$ is uniformly bounded, which contradicts Theorem 1.3 in \cite{Lotay-Wei Shi-estimate}. 
\end{proof}

\begin{theorem}\label{theorem4.5}
     Let $(M,\varphi(t))_{t\in[0,T)}$ be the solution of the Laplacian $G_{2}$ flow $\eqref{The closed Laplacian flow}$ on a closed $7$-dimensional manifold $M$ with $T<+\infty$, where $g(t)$ is the Riemannian metric associated with $\varphi(t)$. Then we have

    $(1)$ either $\displaystyle{\liminf_{t\rightarrow T}\left(\min_{M} R(g(t))\right)=-\infty}$, 

    $(2)$ or $\displaystyle{\liminf_{t\rightarrow T}\min_{M} R(g(t))>-\infty}$, but for any positive constants $C>0$ and $\delta>0$,
    $$\limsup_{t\rightarrow T}\left(\max_{M}|W(g(t))|_{C^{1}(M,g(t))}\right)>\frac{C}{(T-t)^{1-\delta}}.$$
\end{theorem}
\begin{proof}
    We consider case (2) (case (1) can be directly derived by case (2)). We assume that $\displaystyle{\liminf_{t\rightarrow T}\left(\min_{M} R(g(t))\right)>-\infty}$, i.e., 
    $$
    R(g(t))+c\geq1
    $$
    for some universal constant $c$. Since $T<\infty$, the whole Riemann curvature $\rm Rm$ must be blowup, which means $$\limsup_{t\rightarrow T}\left(\max_{M} |W(g(t))|_{C^{1}(M,g(t))}\right)=+\infty.$$
    Now if we assume 
    $$\limsup_{t\rightarrow T}\left(\max_{M}|W(g(t))|_{C^{1}(M,g(t))}\right)\leq\frac{C}{(T-t)^{1-\delta}}$$
    for some constant $C>0$ and $\delta>0$, then by Theorem \ref{theorem4.3}, we get
    \begin{align}
        |{\rm Ric}(g(t))|_{g(t)}\leq c\left[C_{1}+C_{2}\frac{C}{(T-t)^{1-\delta}}\right]\leq\frac{C_{3}}{(T-t)^{1-\delta}}\notag.
    \end{align}
    The Cauchy-Schwarz inequality and the identity $R=-|{\bf T}|^{2}$ yields
    \begin{align}       |\partial_{t}g_{ij}|&=2|S_{ij}|=2\left|R_{ij}+\frac{1}{3}|\mathbf{T}|^{2}g_{ij}+2\mathbf{T}_{ik}\mathbf{T}^{k}_{\ j}\right|\notag\\
        &\leq C_{4}|{\rm Ric}|+C_{5}|R|+C_{6}|\mathbf{T}_{ik}\mathbf{T}^{k}_{\ j}|\notag\\
        &\leq C_{7}|{\rm Ric}|+C_{8}|\mathbf{T}|^{2}\notag\\
        &\leq  C_{9}|{\rm Ric}|\notag.
    \end{align}
    These imply that
    \begin{align}
        \int_{0}^{T}|\partial_{t}g(t)|_{g(t)}dt\leq C_{9}\int_{0}^{T}|{\rm Ric}(g(t))|_{g(t)}dt\leq\int^{T}_{0}\frac{C_{10}}{(T-t)^{1-\delta}}dt:=N<+\infty\notag,
    \end{align}
    From Lemma 6.49 in \cite{Chow-Knopf 2004}, we arrive at
    \begin{align}
        e^{-N}g(0)\leq g(t)\leq e^{N}g(0)\notag
    \end{align}
    for any $t\in [0,T)$. Since the scalar curvature is also uniformly bounded on $[0, T)$, with the same discussion as Theorem 8.1 in \cite{Lotay-Wei Shi-estimate}, the solution of the Laplacian $G_{2}$ flow can be extended past $T$. This is a contradiction.
\end{proof}

\textbf{Acknowledgments.}\ \ 
The authors would like to thank Professor Bing Wang for bringing the reference \cite{Wang 2012}  to our attention. The second author thanks Professor Damin Wu, Professor Xiaokui Yang, and Professor Fangyang Zheng for useful discussions. The first author is supported by the National Key R$\&$D Program of China 2020YFA0712800 and the National Natural Science Foundation of China (NSFC12371049). The second author is funded by the Shanghai Institute for Mathematics and Interdisciplinary Sciences(SIMIS) under grant number SIMIS-ID-2024-LG.  The authors would also like to thank the referee for the valuable comments and suggestions.

\bibliographystyle{amsplain}

\end{document}